\documentclass[12pt, reqno]{amsart}

\usepackage{amsmath,amsfonts,amsbsy,amsgen,amscd,mathrsfs,amssymb,amsthm}
\usepackage[usenames,dvipsnames]{xcolor}

\usepackage[bookmarks=true,
bookmarksnumbered=true, breaklinks=true,linkcolor=blue, citecolor=purple,
pdfstartview=FitH, hyperfigures=false,
plainpages=false, naturalnames=true,
colorlinks=true,pagebackref=true,
pdfpagelabels]{hyperref}
\hypersetup{pdftitle={ },pdfauthor={ }}

\usepackage[a4paper,margin=1in]{geometry}
\usepackage{amsmath}
\allowdisplaybreaks[4]
\numberwithin{equation}{section}
\newtheorem{theorem}{Theorem}[section]

\newtheorem{lemma}[theorem]{Lemma}
\newtheorem{corollary}[theorem]{Corollary}
\newtheorem{remark}[theorem]{Remark}
\newtheorem{proposition}[theorem]{Proposition}

\begin{document}
\begin{sloppypar}

\title[Minkowski problem of anisotropic $p$-torsional rigidity]{Minkowski problem of anisotropic $p$-torsional rigidity}

\author{Chao Li}
\address{Chao Li:
School of Mathematics and Statistics, Ningxia University, Yinchuan, Ningxia, 750021, China.}
\address{Ningxia Basic Science Research Center of Mathematics, Ningxia University, Yinchuan 750021, China.}
	\email{lichao@nxu.edu.cn, lichao166298@163.com}

	
%\date{\today}

\begin{abstract}
In this paper, we consider the Minkowski problem associated with the solution to the anisotropic $p$-Laplacian (or Finsler $p$-Laplacian) equation, namely, the Minkowski problem of anisotropic $p$-torsional rigidity. The sufficient and necessary conditions for the existence of a solution to the Minkowski problem of anisotropic $p$-torsional rigidity are presented. Meanwhile, the existence of a solution to the log-Minkowski problem of anisotropic $p$-torsional rigidity without symmetry assumptions is solved.
\end{abstract}

\maketitle

\vskip 20pt
\section{Introduction and main results}
In the last few years, notable progress has been made in the study of the Minkowski problem. This progress includes the theoretical establishment of the existence, uniqueness, and smoothness of the solutions to the Minkowski problem. The Minkowski problem, initially introduced by Minkowski, was later independently solved by Aleksandrov \cite{AAD1938} and Fenchel and Jessen \cite{FJ1938}. Their work demonstrated that a Borel measure $\mu$ defined on $\mathbb{S}^{n-1}$ is equivalent to the surface area measure of a convex body under the conditions that it is finite, its centroid is at the origin, and is not concentrated on any great subsphere of $\mathbb{S}^{n-1}$. This problem is of great significance in various fields such as differential geometry, convex geometry and partial differential equations.
 
In 1993, Lutwak \cite{LUTWA1993} introduced Firey's $p$-Minkowski sum into the concept volume and proposed a series of far-reaching concepts such as the $L_p$-mixed volume and $L_p$-surface area measure. At the same time, Lutwak \cite{LUTWA1993} also proved the existence and uniqueness of the solution to the $L_p$-Minkowski problem for $p>1$ and $p\neq n$ with even measure. Since then the research on Minkowski problem has attracted considerable attention and produced a number of influential results. In addition, Lutwak and Oliker \cite{LOLI1995} obtained the regularity of solution to the $L_p$-Minkowski problem. Soon afterwards, the $L_p$-Minkowski problem led to the Orlicz-Minkowski problem, Haberl, Lutwak, Yang and Zhang \cite{HLYZ2010} proved the existence of solutions for the even Orlicz-Minkowski problem and provided a new method to solve the classical Minkowski problem. Pioneering work of Huang, Lutwak, Yang and Zhang \cite{HLYZ2016} studied the dual Minkowski problem and gave the existence conditions for the solution of the dual Minkowski problem in the symmetric case. Both the logarithmic Minkowski problem and the Aleksandrov problem are special cases of the dual Minkowski problem. The latest research results on Minkowski problem can be found in the books \cite{BMT2024,BRF2024}, which covers the research contents and methods of Minkowski problem in recent years.

Owing to the profound influence of the Minkowski problem, Minkowski type problems related to the solution of the Laplace equation and inspired by the Minkowski problem emerged. Jerison \cite {JER1989,JER1991} first studied the Minkowski type problem of predetermined harmonic measure. Subsequently, Akman and Mukherjee \cite{AKM2023} studied the Minkowski problem of the $p$-harmonic measure and generalized Jerison's \cite{JER1991} result. Moreover, Li and Zhao \cite{LCZX2024} also proved the existence of smooth solutions to the Minkowski problem of the $p$-harmonic measure. Subsequently, Jerison \cite{JER1996} considered the Minkowski problem of electrostatic capacity measure. Concurrently, Zou and Xiong \cite{ZDX2020} extended the Minkowski problem of electrostatic capacity to the $L_p$ case. Most recently, Dai and Yi \cite{DQY2025} obtained the existence of solutions to Minkowski problems arising from sub-linear elliptic equations. 

In 2010, motivated by the work of Jerison \cite{JER1996}, Colesanti and Fimiani \cite{CAM2010} consider the Minkowski problem of torsional rigidity. Let $u$ be the solution of the boundary-value problem
\begin{align*}
\begin{cases}
\Delta u=-2 & \text { in } K, \\ 
u=0 & \text { on } \partial K.
\end{cases}
\end{align*}
Then the torsional rigidity $\tau(K)$ of an open bounded subset $K$ of $\mathbb{R}^n$ (with some basic boundary regularity) can be defined as
\begin{align*}
\tau(K)=\frac{1}{n+2} \int_{\mathbb{S}^{n-1}} h_K(\xi) |\nabla u(\mathbf{g}^{-1}_K(\xi))|^2d S(K,\xi), 
\end{align*}
the torsional measure, $S_{\tau}(K,\cdot)$, of $K$ is a Borel measure on the unit sphere $\mathbb{S}^{n-1}$ defined for a Borel set $\eta\subset \mathbb{S}^{n-1}$ by
\begin{align*} 
S_{\tau}(K,\eta)=\int_{x\in\mathbf{g}_K^{-1}(\eta)}|\nabla u(x)|^2 d\mathscr{H}^{n-1}(x),
\end{align*}
the torsional measure, $S_{\tau}(K,\cdot)$, of $K$ is absolutely continuous with respect to the surface area measure $S(K,\cdot)$ of $K$. The Minkowski problem related to torsional measure $S_{\tau}(K,\cdot)$, can be stated as follows:

\noindent{\bf The Minkowski problem of torsional rigidity:}~~{\it Let $\mu$ be a finite Borel measure on $\mathbb{S}^{n-1}$. What are necessary and sufficient conditions on $\mu$ to guarantee the existence of a convex body $K \subset \mathbb{R}^n$ containing the origin that solves the equation
$$
S_{\tau}(K,\cdot)=\mu ?
$$}
When posing the Minkowski problem of torsional rigidity, Colesanti and Fimiani \cite{CAM2010} proved the existence and uniqueness of the solution. In 2020, Chen and Dai \cite{CHD2020} obtained the existence of solutions to the $L_p$-Minkowski problem of torsional rigidity for $n + 2\neq p>1$, which extends a result of Colesanti and Fimiani \cite{CAM2010}. Subsequently, Li and Zhu \cite{LNYZ2020} proved the existence of solutions to the Orlicz-Minkowski problem for discrete measures and general measures associated with torsional rigidity. Shortly thereafter, Sun, Xu, and Zhang \cite{SXZ2021} proved the uniqueness of the $L_p$ Minkowski problem for q-torsional rigidity with $p>1$ and $q>1$ in the smooth case. Recently, Hu \cite{HJR2024,HJR20241} established the existence of new smooth solutions to the torsional Minkowski problem under the premise of providing an initial convex body and proved the torsion log-Minkowski problem without symmetry assumptions via an approximation argument.

Over the past several decades, the Laplacian operator has been pivotal in the analysis of fully nonlinear partial differential equations. As scholars have delved deeper into the traditional Laplacian operator, the Finsler $p$-Laplacian or anisotropic $p$-Laplacian has drawn considerable interest due to its distinct characteristics, which allow it to produce the Finsler Laplacian, classical Laplacian, $p$-Laplacian, and pseudo $p$-Laplacian. The origins of research on anisotropic operators can be track to the works of Bellettini and Paolini \cite{BEP1996}. In the realm of fully nonlinear differential equations, the Finsler Laplace operator is used to characterize solution attributes such as regularity and smoothness. As a fundamental instrument, the Finsler Laplace operator aids in the examination of geometric manifold properties. Studies into Finsler Laplacian operators have been widely pursued, as evidenced by the literature, such as in reference \cite{BCG2018,BCS2016,FNS2020,CGL2020,CCPS2009,CSP2023,PDG2014,FNS2018,GRAM024,JAR2014,WGXZ2011,WGX2012,XCY2022,YH2014}.

In this paper, we will consider the Minkowski problem associated with the solution to anisotropic $p$-Laplacian. Let $F^p$ be strictly convex, $K$ be bounded open set of $\mathbb{R}^n$, $1<p<\infty$ and $n \geq 2$, let $u$ be the solution of the boundary-value problem
\begin{align}\label{Eq:YXBJZWT}
\left\{\begin{array}{lll}
\Delta_p^F u=-1 & \text { in } & K, \\
u=0 & \text { on } & \partial K,
\end{array}\right.
\end{align}
where $\Delta_p^F$ is the Finsler $p$-Laplacian (or anisotropic $p$-Laplacian) operator which is given by
\begin{align*}
\Delta_p^F u=\operatorname{div}\left(F^{p-1}(\nabla u) \nabla_{\xi} F(\nabla u)\right),
\end{align*}
then the anisotropic $p$-torsional rigidity $\tau_{F,p}(K)$ of $K$ (the number $\tau_{F,p}(K)>0$) is defined by (see \cite{PDG2014} or Section \ref{SEC3})
\begin{align*}
\tau_{F,p}(K)=\int_{K} F^p\left(\nabla u\right) d x=\int_{K} u d x,
\end{align*}
from this and Anisotropic Poho\v{z}aev identity, we deduce (see Section \ref{SEC3})
\begin{gather*}
\tau_{F,p}(K)=\frac{p-1}{n(p-1)+p} \int_{\mathbb{S}^{n-1}} h_K(\xi) F^p(\nabla u(\mathbf{g}^{-1}_K(\xi)))d S(K,\xi). 
\end{gather*}

If $K$ is a convex body in $\mathbb{R}^n$ that contains the origin in its interior, then the anisotropic $p$-torsional measure, $S_{F,p}(K,\cdot)$, of $K$ is a Borel measure on the unit sphere $\mathbb{S}^{n-1}$ defined for a Borel set $\eta\subset \mathbb{S}^{n-1}$ by
\begin{align*}
S_{F,p}(K,\eta)=\int_{x\in\mathbf{g}_K^{-1}(\eta)}F^p\left(\nabla u(x)\right) d\mathscr{H}^{n-1}(x),
\end{align*}
where $u \in W_0^{1, p}(K) \backslash\{0\}$ and $\mathbf{g}_K: \partial^{\prime} K \rightarrow \mathbb{S}^{n-1}$ is the Gauss map of $K$, defined on $\partial^{\prime} K$, the set of points of $\partial K$ that have a unique outer unit normal. Obviously the anisotropic $p$-torsional measure, $S_{F,p}(K,\cdot)$, of $K$ is absolutely continuous with respect to the surface area measure $S(K,\cdot)$. In addition, the cone anisotropic $p$-torsional measure, $\tau_{F,p}^{\rm{log}}(K,\cdot)$, of $K$ is a Borel measure on the unit sphere $\mathbb{S}^{n-1}$ defined for a Borel set $\eta\subset \mathbb{S}^{n-1}$ by
\begin{gather*}
\tau_{F,p}^{\rm{log}}(K,\eta)=\frac{p-1}{n(p-1)+p}\int_{x\in\mathbf{g}_K^{-1}(\eta)}\langle x, \mathbf{g}_K(x)\rangle F^p\left(\nabla u(x)\right) d\mathscr{H}^{n-1}(x).
\end{gather*} 

There is a close relationship between the anisotropic $p$-torsional rigidity and the classical torsional rigidity. This prompts us to consider the Minkowski problem of anisotropic $p$-torsional rigidity in a similar way.The study of the Minkowski problem related to anisotropic operators can be traced back to the research work of Andrews \cite{ABEND2001}. Afterward, Xia \cite{XC2013} investigated the anisotropic Minkowski problem, that is, prescribing the anisotropic Gauss-Kronecker curvature for a closed strongly convex hypersurface in $\mathbb{R}^{n + 1}$ as a function of its anisotropic normals. In 2022, Akman, Gong, Hineman, Lewis  and Vogel \cite{AGHLV2022} investigated the Minkowski problem associated with nonlinear capacity, which includes the Minkowski problem of anisotropic $p$-capacity. Based on this, we will present and investigate the following Minkowski problem of anisotropic $p$-torsional rigidity.

\noindent{\bf The Minkowski problem of anisotropic $p$-torsional rigidity:}~~{\it 
 Let $1<p<\infty$ and $\mu$ be a Borel measure on $\mathbb{S}^{n-1}$. What properties must $\mu$ have that we may guarantee that it is the anisotropic $p$-torsional measure $S_{F,p}(K, \cdot)$ of some convex body of $\mathbb{R}^n$, i.e.,
$$
S_{F,p}(K, \cdot)=\mu?
$$}

Obviously, when a given finite Borel measure has a positive density $f(\xi)$ on the unit sphere $\mathbb{S}^{n-1}$, $K \subset \mathbb{R}^n$ is a bounded domain with the boundary of class $C^{2, \alpha}$, and $u$ is a unique solution to Problem (\ref{Eq:YXBJZWT}), for $1<p<\infty$, then the existence of Minkowski problem of anisotropic $p$-torsional measure equivently to solving the following Monge-Amp\`{e}re equation
$$
 F^p(\nabla u(\mathbf{g}^{-1}_K(\xi))) \operatorname{det}(h_{i j}+h \delta_{i j})(\xi)=f(\xi),  \quad\xi \in \mathbb{S}^{n-1}.
$$

Firstly, we give the following sufficient and necessary conditions of the existence of solution to the Minkowski problem for anisotropic $p$-torsional rigidity.

\begin{theorem}\label{THm:jdMwt}
 Suppose $1<p<\infty$. If $\mu$ is a non-zero, finite Borel measure on $\mathbb{S}^{n-1}$, then there exists a convex body $K \in \mathscr{K}^n$, so that
$$
S_{F,p}(K,\cdot)=\mu,
$$
if and only if $\mu$ is not concentrated on any closed hemisphere of $\mathrm{S}^{n-1}$, and
\begin{align}\label{Eq:Cxbtj}
\int_{\mathbb{S}^{n-1}} v d \mu(v)=0.
\end{align}
\end{theorem}

Next, the existence of solution to the log-Minkowski problem of anisotropic $p$-torsional rigidity is given in the the case of discrete measures whose supports are in general position. 
\begin{theorem}\label{thmlslamp}
 Let $1<p<\infty$ and $\mu$ be a discrete measure on $\mathbb{R}^n$ whose support set is not contained in any closed hemisphere and is in general position in dimension $n$. Then there exists a polytope $P$ containing the origin in its interior such that
$$
\tau_{F,p}^{\rm{log}}(P, \cdot)=\mu.
$$
\end{theorem}

Finally, the existence of solution to the log-Minkowski problem of anisotropic $p$-torsional rigidity for general convex body is proved by Theorem \ref{thmlslamp}.

\begin{theorem}\label{thmyblmp}
 Let $\mu$ be a Borel measure on $\mathbb{R}^n$ and satisfies the subspace mass inequality for index $2\leq p<\infty$,
\begin{align}\label{Eq:ZKJZL}
\frac{\mu\left(\xi_i \cap \mathbb{S}^{n-1}\right)}{|\mu|}<1-\frac{(n(p-1)+p)(n-i)}{n(n+2)(p-1)},
\end{align}
for each $i$ dimensional subspace $\xi_i \subset \mathbb{R}^n$, and each $i=1, \ldots, n-1$. Then there exists a convex body $K$ containing the origin in its interior such that
$$
\tau_{F,p}^{\rm{log}}(K, \cdot)=\mu.
$$
\end{theorem}

\begin{remark}
Since the classical torsional rigidity is a specific case of the anisotropic $p$-torsional rigidity. When $p = 2$ and $F(\xi) = \sum_k|\xi_k|$ in Theorem \ref{THm:jdMwt} and Theorem \ref{thmyblmp}, we can obtain the result of Colesanti and Fimiani's (see \cite[Theorem 1.1]{CAM2010}) as a particular case of Theorem \ref{THm:jdMwt}, and also the result of Hu's (see \cite[Theorem 1.2]{HJR20241}) as a special case of Theorem \ref{thmyblmp}.
\end{remark}

The paper is organized as follows. In Section \ref{SEC2}, we fix some notations and recall the corresponding results about convex geometry. In Section \ref{SEC3}, we recall some results about the anisotropic $p$-torsional rigidity. Meanwhile, we prove the basic properties of the anisotropic $p$-torsional rigidity and the variational formula for it. In Section \ref{SEC4}, we prove Theorem \ref{THm:jdMwt} using a variational method. Finally, in Section \ref{SEC5}, we provide the proof of Theorem \ref{thmlslamp} and then employ the approximation method to prove Theorem \ref{thmyblmp}.

\section{Notations and Background Materials}\label{SEC2}

\subsection{Convex Geometry}

Let $\mathscr{K}^n$ denote the set of convex bodies (compact, convex subsets with non-empty interiors) in Euclidean space $\mathbb{R}^n$. Let $\mathscr{K}^n_o$ denote the set of convex bodies containing the origin in their interiors in $\mathbb{R}^n$. For the set of star bodies (about the origin) in $\mathbb{R}^n$, we write $\mathscr{S}^n_o$. Besides, let $\mathbb{S}^{n-1}$ denote the unit sphere, $\mathbb{B}^{n}$ denote the unit ball and $|K|$ denote the $n$-dimensional volume of $K$.

First, let's introduce some basic knowledges about polytope (see \cite{ZGX2014,LEE2021,SRA2014}). Let $\alpha_1,\cdots,\alpha_N$ be positive constants and  $\delta_{u_k}$ denotes the delta measure that is concentrated at the point $u_k$, then the discrete measure $\mu$ is expressed as 
\begin{align}\label{Eq:sjj2}
\mu=\sum_{k=1}^N \alpha_k \delta_{u_k}.
\end{align}

A subset $U$ of $\mathbb{S}^{n - 1}$ lies within a closed hemisphere if there exists a vector $u \in \mathbb{S}^{n - 1}$ such that $U$ is contained in $\{v \in \mathbb{S}^{n - 1} : u \cdot v \geq 0\}$. 

A finite subset $U$ of $\mathbb{S}^{n - 1}$ (with no less than $n$ elements) is in general position if any $n$ elements of $U$ are linearly independent.

A polytope in $\mathbb{R}^n$ is the convex hull of a finite set of points with positive $n$-dimensional volume in $\mathbb{R}^n$. If the convex hull of a certain subset lies on the boundary of the polytope  and has positive $(n - 1)$-dimensional volume, it is called a face of the polytope .

Let the unit vectors $u_1,\cdots,u_N \in \mathbb{S}^{n - 1}$ not be concentrated on any closed hemisphere. Denote $\mathcal{P}(u_1,\cdots,u_N)$ as the set of polytope 
$$P = \bigcap_{i = 1}^N \{x \in \mathbb{R}^n : x \cdot u_i \leq a_i\},$$
where $a_1,\cdots,a_N \in \mathbb{R}$. Let $\mathcal{P}_N(u_1,\cdots,u_N)$ be a subset of $\mathcal{P}(u_1,\cdots,u_N)$, and its elements have exactly $N$ faces. Among them, for the polytope  $P$ in $\mathcal{P}(u_1,\cdots,u_N)$, the maximum number of its faces is $N$, and the outer normal vectors of $P$ are a subset of $\{u_1,\cdots,u_N\}$. 

A convex body $K\in \mathscr{K}^n$, is uniquely determined by its support function, $h(K, \cdot)=h_K:\mathbb{R}^n\rightarrow\mathbb{R}$, which is defined by (see \cite{SRA2014})
\begin{align}\label{SF}
h(K, x)=\max\{ \langle x, y\rangle : y\in K\}, \ \ \ \ x\in \mathbb{R}^n,
\end{align}
where $\langle x, y\rangle$ denotes the standard inner product of $x$ and $y$ in $\mathbb{R}^n$. It is also clear from the definition that $h(K,u) \leq h(L,u)$ for $u \in \mathbb{S}^{n-1}$, if and only if $K\subseteq L$.

For $K,L\in\mathscr{K}_o^n$, and $\alpha,\beta\geq0$ (not both zero), the Minkowski combination, $\alpha\cdot K+ \beta\cdot L$, of $K$ and $L$ is defined by (see \cite{SRA2014})
\begin{align}\label{lpmz}
h(\alpha\cdot K+ \beta\cdot L,\cdot)=\alpha h(K,\cdot)+\beta h(L,\cdot),
\end{align}
where $``+"$ denotes the Minkowski-sum and $\alpha\cdot K=\alpha K$ is the Minkowski scalar multiplication.

If $K$ is a compact star-shaped set (about the
origin) in $\mathbb{R}^n$, then its radial function,
$\rho_K=\rho(K,\cdot):\mathbb{R}^n\setminus\{0\}\rightarrow[0,\infty)$, is defined by
\begin{align}\label{RF}
\rho(K,x)=\max\{c\geq0: c x\in K\}, \ \ x\in\mathbb{R}^n\setminus \{0\}.
\end{align}
If $\rho_K$ is positive and continuous, $K$ will be called a star body (with respect to the origin).

The radial map $r_K: \mathbb{S}^{n-1} \rightarrow \partial K$ is
$$
r_K(\xi)=\rho_K(\xi) \xi,
$$
for $\xi \in \mathbb{S}^{n-1}$, i.e. $r_K(\xi)$ is the unique point on $\partial K$ located on the ray parallel to $\xi$ and emanating from the origin.

Using this we see that the outer unit normal vector to $\partial K$ at $x$, denoted by $\mathbf{g}(x)$, is well defined for $\mathscr{H}^{n-1}$ almost all $x \in \partial K$. The map $\mathbf{g}: \partial K \rightarrow \mathbb{S}^{n-1}$ is called the Gauss map of $K$. For $\omega \subset \mathbb{S}^{n-1}$, let
$$
\mathbf{g}^{-1}(\omega)=\{x \in \partial K: \mathbf{g}(x) \text { is defined and } \mathbf{g}(x) \in \omega\}.
$$

If $\omega$ is a Borel subset of $\mathbb{S}^{n-1}$, then $\mathbf{g}^{-1}(\omega)$ is $\mathscr{H}^{n-1}$-measurable. The Borel measure $S_K$, on $\mathbb{S}^{n-1}$, is defined for Borel $\omega \subset \mathbb{S}^{n-1}$ by
$$
S_K(\omega)=\mathscr{H}^{n-1}(\mathbf{g}^{-1}(\omega)), 
$$
which is called the surface area measure of $K$. For every $f \in C(\mathbb{S}^{n-1})$,
\begin{align}\label{Eq:gdfs}
\int_{\mathbb{S}^{n-1}} f(\xi) d S_K(\xi)=\int_{\partial K} f(\mathbf{g}(x)) d \mathscr{H}^{n-1}(x).
\end{align}

Let $C^+(\mathbb{S}^{n-1})$ denote a positive continuous function on $\mathbb{S}^{n-1}$, and $E$ be a closed subset of $\mathbb{S}^{n-1}$ that is not contained in a closed hemisphere, if $f\in C^+(\mathbb{S}^{n-1})$, the Wulff shape $[f]$ associated with $f$ is defined as
\begin{align*}
[f]=\left\{x \in \mathbb{R}^{n}: x \cdot u \leq f(u), \text { for all } u \in E\right\}.
\end{align*}
Obviously, $h_{[f]}(u)\leq h(u).$

\begin{lemma}[Porposition 2.5 of \cite{KLL2023}]\label{lem:jxy1}
Let $K \in \mathscr{K}_0^n$, the Jacobian of $r_K: \mathbb{S}^{n-1} \rightarrow \partial K$ is $\frac{\rho_K^n(v)}{h_K\left(\mathbf{g}_K\left(r_K(v)\right)\right)}$ up to a set of measure zero.
\end{lemma}

\begin{lemma}[Lemma 2.6 of \cite{KLL2023}]\label{lem:jxy}
 Let $K \in \mathscr{K}_0^n$ and $f \in C(\mathbb{S}^{n-1})$. Then there exists a small $\delta>0$ and $M>0$ such that, for $|t|<\delta$:

(1) $h_t(v)=h_K(v)+t f(v)$ is positive for all $v \in \mathbb{S}^{n-1}$,

(2) $\left|\rho_{\left[h_t\right]}(v)-\rho_K(v)\right|<M|t|$ for almost all $v \in \mathbb{S}^{n-1}$.

Additionally, one has, for almost all $v \in \mathbb{S}^{n-1}$ up to a set of spherical Lebesgue measure zero, that
$$
\left.\frac{\mathrm{d} \rho_{\left[h_t\right]}(v)}{\mathrm{d} t}\right|_{t=0}=\lim _{t \rightarrow 0} \frac{\rho_{\left[h_t\right]}(v)-\rho_K(v)}{t}=\frac{f\left(\mathbf{g}_K\left(r_K(v)\right)\right)}{h_K\left(\mathbf{g}_K\left(r_K(v)\right)\right)} \rho_K(v).
$$
\end{lemma}

The following concepts from metric theory are employed. Given points $x, y \in \mathbb{R}^n$ and a non-empty subset $A \subseteq \mathbb{R}^n$, the expression $|x - y|$ denotes the Euclidean distance between $x$ and $y$, while  
$$
\text{dist}(x, A) := \inf\{ |x - a| : a \in A \}
$$
represents the distance from the point $x$ to the set $A$.

The Hausdorff distance of the sets $K, L$ is nonempty compact subsets of $\mathbb{R}^n$ is defined by
$$
d_H(K, L):=\max \left\{\sup _{x \in K} \inf _{y \in L}|x-y|, \sup _{x \in L} \inf _{y \in K}|x-y|\right\}
$$
or, equivalently, by
\begin{align}\label{EQ:HDJL}
d_H(K, L)=\min \left\{\lambda \geq 0: K \subset L+\lambda \mathbb{B}^{n}, L \subset K+\lambda \mathbb{B}^{n}\right\}.
\end{align}
Then $d_H$ is the Hausdorff metric.

\subsection{Norms of $\mathbb{R}^n$}\label{subH}

Let $F: \mathbb{R}^n \mapsto[0,+\infty)$ be a convex function of class $C^2\left(\mathbb{R}^n \backslash\{0\}\right)$, such that

(i) $F$ is convex,

(ii) $F(\xi) \geq 0$ for $\xi \in \mathbb{R}^n$ and $F(\xi)=0$ if and only if $\xi=0$,

(iii) $F(t \xi)=|t| F(\xi)$ for $\xi \in \mathbb{R}^n$ and $t \in \mathbb{R}$.

Obviously, $F^2$ is strictly convex in $\mathbb{R}^n \backslash\{0\}$. A typical example is $F(\xi)=\left(\sum_{i=1}^n\left|\xi_i\right|^p\right)^{1 / p}$ for $p \in(1, \infty)$. For such a function $F$, there exist two constants $0<a \leq b<\infty$ such that
\begin{align}\label{Eq:Fyjd}
a|\xi| \leq F(\xi) \leq b|\xi|, \quad \text { for any } \xi \in \mathbb{R}^n.
\end{align}

More precisely we identify the dual space of $\mathbb{R}^n$ with $\mathbb{R}^n$ itself via the scalar product $\langle\cdot,\cdot\rangle$. Accordingly the space $\mathbb{R}^n$ turns out to be endowed with the dual norm $F_0$ given by
$$
F_0(x)=\sup _{\xi \neq 0} \frac{\langle x, \xi\rangle}{F(\xi)}, \quad \text { for any } x \in \mathbb{R}^n.
$$
On the other hand, we can define $F$ in terms of $F_0$ as
$$
F(\xi)=\sup _{x \neq 0} \frac{\langle x, \xi\rangle}{F_0(x)},  \quad \text { for any }  \xi \in \mathbb{R}^n.
$$

\begin{lemma}[Lemma 2.2 of \cite{JAR2014}]\label{lemhdy}
 If $1<p<\infty$ and $F(x)$ is an arbitrary norm in $\mathbb{R}^n$ which is of class $C^1$ for $x \neq 0$, then
$$
\left|F(\eta)^p-F(\xi)^p\right| \leqslant p[F(\xi)+F(\eta)]^{p-1} F(\eta-\xi),
$$
for any $\xi, \eta \in \mathbb{R}^n$.
\end{lemma}

\section{Anisotropic $p$-torsional rigidity}\label{SEC3}

\subsection{Anisotropic $p$-torsional rigidity}

In 2014, Pietra and Gavitone \cite{PDG2014} considered the following problem: Let $F^p$ be strictly convex, $K$ be bounded open set of $\mathbb{R}^n$, $1<p<\infty$ and $n \geq 2$,
\begin{align}\label{Eq:patr}
\left\{\begin{array}{lll}
\Delta_p^F u=-1 & \text { in } & K, \\
u=0 & \text { on } & \partial K,
\end{array}\right.
\end{align}
where $\Delta_p^F$ is the Finsler $p$-Laplacian (or anisotropic $p$-Laplacian) operator which is given by
\begin{align}\label{Eq:fsplls}
\Delta_p^F u=\operatorname{div}\left(F^{p-1}(\nabla u) \nabla_{\xi} F(\nabla u)\right),
\end{align}
in the sense of distributions; more precisely, (\ref{Eq:patr}) reads as
$$
\int_{K} F^{p-1}(\nabla u)\left\langle\nabla_{\xi} F( \nabla u), \nabla \phi\right\rangle d x=\int_{K} \phi d x,
$$
for any $\phi \in C_0^1(K)$.

Pietra and Gavitone \cite{PDG2014} obtained the solution is unique in (\ref{Eq:patr}). Meanwhile, Pietra and Gavitone \cite{PDG2014} defined the anisotropic $p$-torsional rigidity of $K$ (the number $\tau_{F,p}(K)>0$) such that
\begin{align}\label{Eq:FPTR}
\tau_{F,p}(K)=\int_{K} F^p\left(\nabla u\right) d x=\int_{K} u d x,
\end{align}
where $u \in W_0^{1, p}(K)$ is the unique solution of (\ref{Eq:patr}).

By (\ref{Eq:fsplls}), the Finsler or anisotropic $p$-Laplacian of a function $u$ of class $C^2$ can be written as
\begin{align}\label{Eq:Flplsz}
\Delta_p^F u=F^{p-2}(\nabla u)\left[(p-1) F_{\xi_i}(\nabla u)F_{\xi_j}(\nabla u)+F(\nabla u) F_{\xi_i \xi_j}(\nabla u)\right] u_{i j},
\end{align}

In addition, Pietra and Gavitone \cite{PDG2014} characterized the anisotropic $p$-torsional rigidity, i.e., $\tau_{F,p}(K)=\sigma(K)^{\frac{1}{p-1}}$, where $\sigma(K)$ is the best constant in the Sobolev inequality $\|u\|_{L^1(K)}^p \leq \sigma(K)\|F(\nabla u)\|_{L^p(K)}^p$, that is
\begin{align}\label{Eq:patm}
\tau_{F,p}(K)=\max _{\substack{\psi \in W_0^{1, p}(K) \backslash\{0\} \\ \psi \geq 0}} \bigg(\frac{\left(\int_{K} \psi d x\right)^p}{\int_{K} F(\nabla \psi)^p d x}\bigg)^{\frac{1}{p-1}},
\end{align}
and the solution $u$ of (\ref{Eq:patr}) realizes the maximum in (\ref{Eq:patm}).

At the same time, Pietra and Gavitone \cite[Remark 2.1 and Remark 5.2]{PDG2014} used the anisotropic P\'{o}lya-Szeg\"{o} inequality characterized the upper bound and lower bound estimate of the anisotropic $p$-torsional rigidity.
\begin{lemma}[\cite{PDG2014}]\label{lemuppb} 
Let $1<p<+\infty$ and $K$ be a bounded open set of $\mathbb{R}^n$. Then
\begin{align}\label{Eq:sjgj}
\tau_{F,p}(K) \leq\frac{p-1}{n(p-1)+p}n^{-\frac{1}{p-1}}\kappa_n^{-\frac{p}{n(p-1)}}|K|^{\frac{n(p-1)+p}{n(p-1)}},
\end{align}
where $\kappa_n=|\mathcal{W}|$ and $|\mathcal{W}|$ denotes the Lebesgue measure of $\mathcal{W}$. The set
$$
\mathcal{W}=\left\{\xi \in \mathbb{R}^n: H^o(\xi)<1\right\},
$$
is the so-called Wulff shape centered at the origin. 
\end{lemma}

\begin{lemma}[\cite{PDG2014}]\label{lemxjgj}
Let $1<p<+\infty$ and $K$ be a bounded convex open set of $\mathbb{R}^n$. Then
\begin{align}\label{Eq:xjgj}
\tau_{F,p}(K)\geq \frac{p-1}{2 p-1} \frac{|K|^{\frac{2 p-1}{p-1}}}{C P(K)^{\frac{p}{p-1}}},
\end{align}
where $C$ is a positive constant and $P(K)$ is the perimeter of $K$.
\end{lemma}

Bianchin and Ciraolo \cite{BCG2018} proved the following Poho\v{z}aev identity.

\begin{lemma}[Anisotropic Poho\v{z}aev identity] Let $1<p<\infty$ and $u$ be the solution to (\ref{Eq:patr}), then
\begin{align}\label{Eq:API}
(n(p-1)+p) \int_{K} u(x) d x=(p-1) \int_{\partial K} F^p(\nabla u(x))\langle x, \mathbf{g}_K(x)\rangle d \mathscr{H}^{n-1}(x).
\end{align}
\end{lemma}

\begin{proposition} Let $1<p<\infty$, $K$ be a convex body with boundary of class $C^2$ such that the Gauss curvature is positive at every point of their boundary and $u$ be the solution of problem (\ref{Eq:patr}) in $K$. The following formula holds
\begin{gather}\label{Eq:AFNZg}
\tau_{F,p}(K)=\frac{p-1}{n(p-1)+p} \int_{\mathbb{S}^{n-1}} h_K(\xi) F^p(\nabla u(\mathbf{g}^{-1}_K(\xi)))d S(K,\xi).
\end{gather}
\end{proposition}
\begin{proof}
Since $u$ is the solution of problem (\ref{Eq:patr}) in $K$, from (\ref{Eq:FPTR}), Anisotropic Poho\v{z}aev identity (\ref{Eq:API}) and (\ref{Eq:gdfs}), we can easily deduce that formula (\ref{Eq:AFNZg}) holds.
\end{proof}

\begin{remark}
If $p=2$ and $F(\xi)=\sum_k\left|\xi_k\right|$ in (\ref{Eq:patr}) and (\ref{Eq:AFNZg}), we can obtain the classical torsional rigidity defined by Colesanti (see \cite{COA2005}) as a special case.
\end{remark}

If $1<p<\infty$ and $K$ is a convex body in $\mathbb{R}^n$ that contains the origin in its interior, then the anisotropic $p$-torsional measure, $S_{F,p}(K,\cdot)$, of $K$ is a Borel measure on the unit sphere $\mathbb{S}^{n-1}$ defined for a Borel set $\eta\subset \mathbb{S}^{n-1}$ by
\begin{align}\label{Eq:AMBJCD}
S_{F,p}(K,\eta)=\int_{x\in\mathbf{g}_K^{-1}(\eta)}F^p\left(\nabla u(x)\right) d\mathscr{H}^{n-1}(x),
\end{align}
where $u \in W_0^{1, p}(K) \backslash\{0\}$ and $\mathbf{g}_K: \partial^{\prime} K \rightarrow \mathbb{S}^{n-1}$ is the Gauss map of $K$, defined on $\partial^{\prime} K$, the set of points of $\partial K$ that have a unique outer unit normal. Obviously, the anisotropic $p$-torsional measure, $S_{F,p}(K,\cdot)$, of $K$ is absolutely continuous with respect to the surface area measure $S(K,\cdot)$ of $K$. 

From (\ref{Eq:AFNZg}), if $1<p<\infty$ and $K$ is a convex body in $\mathbb{R}^n$ that contains the origin in its interior, then the cone anisotropic $p$-torsional measure, $\tau_{F,p}^{\rm{log}}(K,\cdot)$, of $K$ is a Borel measure on the unit sphere $\mathbb{S}^{n-1}$ defined for a Borel set $\eta\subset \mathbb{S}^{n-1}$ by
\begin{gather}\label{Eq:logcd}
\tau_{F,p}^{\rm{log}}(K,\eta)=\frac{p-1}{n(p-1)+p}\int_{x\in\mathbf{g}_K^{-1}(\eta)} F^p\left(\nabla u(x)\right) \langle x, \mathbf{g}_K(x)\rangle d\mathscr{H}^{n-1}(x).
\end{gather}

\subsection{Properties of anisotropic $p$-torsional rigidity}

Next we will give some properties of anisotropic $p$-torsional rigidity $\tau_{F,p}$.

\begin{proposition}\label{prop:apxz}
Suppose $u$ is the solution of (\ref{Eq:patr}), $1<p<\infty$ and $K, L \in \mathscr{K}^n$. 
\begin{itemize}
  \item[\bf(a)] The anisotropic $p$-torsional rigidity $\tau_{F,p}(K)$ is positively homogeneous of order $\frac{p}{p-1}+n$, i.e., 
$$\tau_{F,p}(\lambda K)=\lambda^{\frac{p}{p-1}+n}\tau_{F,p}(K).$$

  \item[\bf(b)] The anisotropic $p$-torsional rigidity $\tau_{F,p}(K)$ is translation invariant, i.e., $\tau_{F,p}(K+x)=\tau_{F,p}(K)$ for $K \in \mathscr{K}^n$, $x \in \mathbb{R}^n$.

  \item[\bf(c)] The anisotropic $p$-torsional measure $S_{F,p}(K,\cdot)$ is translation invariant.

  \item[\bf(d)] If $K, L \in \mathscr{K}^n$ and $K \subseteq L$, then $\tau_{F,p}(K) \leq \tau_{F,p}(L)$.
\end{itemize}
\end{proposition}
\begin{proof}
{\bf(a)} If $u$ is the solution of (\ref{Eq:patr}) in $K$ and $\lambda > 0$, by (\ref{Eq:Flplsz}), then the function
\begin{align}\label{Eq:uv28qwe}
v(y)=\lambda^\frac{p}{p-1} u\left(\frac{y}{\lambda}\right), \quad y \in \lambda K,
\end{align}
is the corresponding solution in $\lambda K=\{y=\lambda x \mid x \in K\}$.

Let $u(y)$ be the unique solution of (\ref{Eq:patr}) in $K$, by (\ref{Eq:uv28qwe}) then the function $v(y)=\lambda^\frac{p}{p-1} u\left(\frac{y}{\lambda}\right)$, $y \in \lambda K$, is the unique solution of (\ref{Eq:patr}) in $\lambda K$. Thus, by (\ref{Eq:FPTR}), we deduce
\begin{align*}
\tau_{F,p}(\lambda K)=\int_{\lambda K}v(y)d y=\int_{\lambda K}\lambda^\frac{p}{p-1} u\left(\frac{y}{\lambda}\right)d y=\lambda^{\frac{p}{p-1}+n}\int_{K} u(x)d x=\lambda^{\frac{p}{p-1}+n}\tau_{F,p}(K).
\end{align*}

{\bf(b)} Support $v(\xi)=u(\xi-x)$, $x \in \mathbb{R}^n$, then $v$ is solution of (\ref{Eq:patr}) in $K+x$. From (\ref{Eq:FPTR}), since $u(y) \in W_0^{1, p}(K)$ is the unique solution of (\ref{Eq:patr}). Then from (\ref{Eq:FPTR}), we have
$$\tau_{F,p}(K+x)=\int_{K+x} v(\xi) d \xi=\int_{K} u(y) d y=\tau_{F,p}(K).$$ 
Thus, $\tau_{F,p}(K)$ is translation invariant.

{\bf(c)} From {\bf(b)} we see that $\tau_{F,p}(K+x)=\tau_{F,p}(K)$, for $x \in \mathbb{R}^n$, which implies that $S_{F,p}(K+x, \cdot)=S_{F,p}(K, \cdot)$, that is, the anisotropic $p$-torsional measure $S_{F,p}(K,\cdot)$ is translation invariant.

{\bf(d)} Since the solution $u$ of (\ref{Eq:patr}) realizes the maximum in (\ref{Eq:patm}). If $u$ is the solution of (\ref{Eq:patr}), $K, L \in \mathscr{K}^n$ and $K \subseteq L$. Suppose $f \in W_0^{1, p}\left(K\right)$, then $f=0$ in $\mathbb{R}^n \backslash K$. Obviously, $f \in W_0^{1, p}\left(L\right)$, since $f=0$ in $L \backslash K$. If $f \in W_0^{1, p}\left(L\right)$, then $f=0$ in $\mathbb{R}^n \backslash L$, but we don't know whether $f=0$ in $L \backslash K$, so we can't derive $f \in W_0^{1, p}\left(K\right)$ from $f \in W_0^{1, p}\left(L\right)$. Therefore, $W_0^{1, p}\left(K\right) \subseteq W_0^{1, p}\left(L\right)$ if $K \subseteq L$.
Then
$$
\begin{aligned}
\tau_{F,p}(K)&=\max _{\substack{u \in W_0^{1, p}(K) \backslash\{0\} \\ u \geq 0}} \bigg(\frac{\left(\int_{K} u d x\right)^p}{\int_{K} F(\nabla u)^p d x}\bigg)^{\frac{1}{p-1}}\\
&=\max _{\substack{u \in W_0^{1, p}(K) \backslash\{0\} \\ u \geq 0}} \bigg(\frac{\left(\int_{L} u d x\right)^p}{\int_{L} F(\nabla u)^p d x}\bigg)^{\frac{1}{p-1}}\\
&\leq\max _{\substack{u \in W_0^{1, p}(L) \backslash\{0\} \\ u \geq 0}} \bigg(\frac{\left(\int_{L} u d x\right)^p}{\int_{L} F(\nabla u)^p d x}\bigg)^{\frac{1}{p-1}}\\
& =\tau_{F,p}(L).
\end{aligned}
$$
Thus,
$$
\tau_{F,p}(K) \leq \tau_{F,p}(L).
$$
The proof is completed.
\end{proof}

\begin{remark}
It is evident that when $p=2$ and $F(\xi)=\sum_k\left|\xi_k\right|$ in Proposition \ref{prop:apxz}, these characteristics align with the classical properties of the classical torsional rigidity.
\end{remark}

Let $F$ be a norm in the class
$$
\mathscr{I}_p=\left\{F \in C^{2, \alpha}\left(\mathbb{R}^n \backslash\{0\}\right): \quad \frac{1}{p} F^p  \in C_{+}^2\left(\mathbb{R}^n \backslash\{0\}\right)\right\},
$$
with $p>1$ and for some $\alpha \in(0,1)$, that is a regular norm $F$, where $F^p$ is a twice continuously differentiable function in $\mathbb{R}^n \backslash\{0\}$ whose Hessian matrix is strictly positive definite.

The following lemma was established by Bianchin and Ciraolo (see \cite[Lemma 3.1]{BCG2018}  and its proof).
\begin{lemma}[\cite{BCG2018}]\label{lemia}
Let $K \subset \mathbb{R}^n$ be a bounded domain with the boundary of class $C^{2, \alpha}$ and let $F$ be a norm in $\mathcal{I}_p$ and $p>1$. There exists a unique solution $u$ to Problem (\ref{Eq:patr}) with $u \in C^{1, \alpha}(\bar{K})$. Meanwhile, if $\partial K$ is of class $C^{2, \alpha}$, then there exists $0<c<1$ such that 
\begin{align}\label{Eq:tdgij1}
c \leq|\nabla u| \leq \frac{1}{c} \text{ on }  \partial K.
\end{align}
\end{lemma}

\begin{lemma}\label{lemfqjix}
Let $1<p<\infty$ and $K \subset \mathbb{R}^n$ be bounded domain with the boundary of class $C^{2, \alpha}$ and $F$ be a norm in $\mathscr{I}_p$.  Let $u$  be the solution to equation (\ref{Eq:patr}) on $K$.  Then for every point  $x \in \partial K$, the gradient $\nabla u$ has a finite non-tangential limit at $x$.
\end{lemma}
\begin{proof}
Since $K\subset \mathbb{R}^n$ is a bounded domain with the boundary of class $C^{2, \alpha}$, $F$ is a norm in $\mathscr{I}_p$ and $p>1$. By Lemma \ref{lemia}, there exists a unique solution $u$ to Problem (\ref{Eq:patr}) with $u \in C^{1, \alpha}(\bar{K})$. This implies that $\nabla u$ is uniformly continuous on $\bar{K}$. In particular, for any fixed $x \in \partial K$ and any $y \in \Gamma_\beta(x) \subset K$, given $\varepsilon > 0$, choose $\delta > 0$ and  constant $C$, such that $\delta = \left(\frac{\varepsilon}{C}\right)^{1/\alpha}$. Then for any $y \in \Gamma_\beta(x)$ with $0 < |y - x| < \delta$, we obtain
$$
|\nabla u(y) - \nabla u(x)| \leq C |y - x|^\alpha < C \cdot \delta^\alpha = \varepsilon. 
$$
Therefore,
$$
\lim_{\substack{y \to x \\ y \in \Gamma_\beta(x)}} \nabla u(y) = \nabla u(x).
$$
\end{proof}

\begin{lemma}\label{lemyzyj}
Let $1<p<\infty$ and $K,K_i \subset \mathbb{R}^n$ be bounded domain with the boundary of class $C^{2, \alpha}$ and $F$ be a norm in $\mathscr{I}_p$. If $K_i$ converges to $K$ in Hausdorff metric as $i \rightarrow \infty$.  Let $u$ and $u_i$ be the solutions to equation (\ref{Eq:patr}) on $K$ and $K_i$, respectively.
Then $u_i \to u$ uniformly in $C^1(\bar{K_i})$.
\end{lemma}
\begin{proof}
Since $K_i$ converges to $K$ in Hausdorff metric as $i \rightarrow \infty$, from (\ref{EQ:HDJL}), there exists a sequence $\{\epsilon_i\}_{i=1}^\infty$ with $\epsilon_i \to 0$  as $i \to \infty$ such that 
\begin{align}\label{Eq:Hsdc}
K \subset K_i+\epsilon_i \mathbb{B}^n \quad \text{and} \quad K_i \subset K+\epsilon_i \mathbb{B}^n,
\end{align}
and there exists a compact set $\Omega \subset \mathbb{R}^n$ such that   
$$
\bigcup_{i \geq 1} \bar{K_i} \cup \bar{K} \subset \Omega.
$$ 
Since $\partial K, \partial K_i \in C^{2,\alpha}$, there exists a finite open cover $\{B_j\}_{j=1}^N$ and local coordinates $(x', x_n) \in \mathbb{R}^{n-1} \times \mathbb{R}$ such that for each $j$, there exist $\gamma_j \in C^{2,\alpha}(B_j')$ and $\gamma_j^i \in C^{2,\alpha}(B_j')$ (where $B_j' \subset \mathbb{R}^{n-1}$  is the projection of $B_j$ onto $\mathbb{R}^{n-1}$) satisfying  
$$
\partial K \cap B_j=\{(x', \gamma_j(x')) \mid x' \in B_j'\}, \quad \partial K_i \cap B_j = \{(x', \gamma_j^i(x')) \mid x' \in B_j'\},
$$
 with $\|\gamma_j^i - \gamma_j\|_{C^{2,\alpha}(B_j')} \to 0$ as $i \to \infty$.

By \cite[Lemma 3.1]{BCG2018} there exists a unique solution $u_i \in C^{1,\alpha}(\bar{K_i})$  to Problem (\ref{Eq:patr}) on $K_i$ satisfying
\begin{align}\label{eq:jbgji1}
\|u_i\|_{C^{1,\alpha}(\bar{K_i})} \leq C_1,
\end{align}
where  $C_1$ is constant independent of $i$.  
 
For each $j$, by reflection extension define $\tilde{u}_i^j \in C^{1, \alpha}\left(B_j\right)$ by
\begin{align}\label{Eq:jbyt}
\tilde{u}_i^j(x)= \begin{cases}u_i(x), & x_n \geq \gamma_j^i\left(x^{\prime}\right), \\ -u_i\left(x^{\prime}, 2 \gamma_j^i\left(x^{\prime}\right)-x_n\right), & x_n<\gamma_j^i\left(x^{\prime}\right),\end{cases}
\end{align}
which satisfies 
\begin{align}\label{eq:jbgji2}
\left\|\tilde{u}_i^j\right\|_{C^{1, \alpha}\left(B_j\right)} \leq C_2\left\|u_i\right\|_{C^{1, \alpha}\left(B_j \cap K_i\right)} \leq C_2 C_1,
\end{align}
where $C_2$ is independent of $i$. 

Let $\{\psi_j\}$ be a partition of unity subordinate to $\{B_j\}$, where $0 \leq \psi_j \leq 1$ and $\sum_{j=1}^N \psi_j =1$ in  $\bigcup_{j=1}^N B_j$. Combine (\ref{Eq:jbyt}), define the global extension $\tilde{u}_i: \Omega \rightarrow \mathbb{R}$ by
$$
\tilde{u}_i(x)= \begin{cases}\sum_{j=1}^N \psi_j(x) \tilde{u}_i^j(x), & x \in \bigcup_{j=1}^N B_j, \\ 0, & x \in \Omega \backslash \bigcup_{j=1}^N B_j.\end{cases}
$$
Since $\operatorname{supp}\left(\psi_j\right) \subset B_j$ and $\sum \psi_j=1$ in a neighborhood of $\partial K_i$, it follows that $\tilde{u}_i \in C^{1, \alpha}(\Omega)$ and $\tilde{u}_i=u_i$ on $K_i$. Moreover, by (\ref{eq:jbgji1}) and (\ref{eq:jbgji2}), we have
\begin{align}\label{eq:jbgji3}
\left\|\tilde{u}_i\right\|_{C^{1, \alpha}(\Omega)} \leq C_3 \sum_{j=1}^N\left\|\tilde{u}_i^j\right\|_{C^{1, \alpha}\left(B_j\right)} \leq C',
\end{align}
where $C'$ is independent of $i$.

The (\ref{eq:jbgji3}) implies
\begin{align}\label{Eq:dosa}
\|\tilde{u}_i\|_{L^\infty(\Omega)} \leq C', \quad \|\nabla \tilde{u}_i\|_{L^\infty(\Omega)} \leq C'.
\end{align}
Furthermore, since $\tilde{u}_i \in C^{1,\alpha}(\Omega)$, its gradient $\nabla \tilde{u}_i$ is $\alpha$-H\"{o}lder continuous, thus
\begin{align}\label{Eq:dosa2}
[\nabla \tilde{u}_i]_{C^{0,\alpha}(\Omega)} = \sup_{\substack{x \neq y \\ x,y \in \Omega}} \frac{|\nabla \tilde{u}_i(x) - \nabla \tilde{u}_i(y)|}{|x - y|^\alpha} \leq C'.
\end{align}
From (\ref{Eq:dosa}) and the mean value theorem yields, for any $x, y \in \Omega$,
$$
|\tilde{u}_i(x) - \tilde{u}_i(y)| \leq \|\nabla \tilde{u}_i\|_{L^\infty(\Omega)} |x - y| \leq C' |x - y|,
$$
choosing $\delta_1 = \epsilon / C'$ and $\epsilon > 0$, ensures that $|x - y| < \delta_1$ implies $|\tilde{u}_i(x) - \tilde{u}_i(y)| < \epsilon$ for all $i$. This mean that $\{\tilde{u}_i\}$ is equicontinuous.

In addition, (\ref{Eq:dosa2}) guarantees that for any $x, y \in \Omega$,
$$
|\nabla \tilde{u}_i(x) - \nabla \tilde{u}_i(y)| \leq [\nabla \tilde{u}_i]_{C^{0,\alpha}(\Omega)} |x - y|^\alpha \leq C' |x - y|^\alpha,
$$
let $\delta_2 = (\epsilon / C')^{1/\alpha}$ and  $\epsilon > 0$, ensures that $|x - y| < \delta_2$ implies $|\nabla \tilde{u}_i(x) - \nabla \tilde{u}_i(y)| < \epsilon$ for all $i$. Thus $\{\nabla \tilde{u}_i\}$ is also equicontinuous.

By Arzel\`{a}-Ascoli theorem there exists a subsequence still denoted $\{\tilde{u}_i\}$ such that 
$\tilde{u}_i \rightarrow \tilde{u}$ uniformly in $C^0(\Omega)$ and there exists a further subsequence still denoted $\left\{\nabla \tilde{u}_i\right\}$ such that $\nabla \tilde{u}_i \rightarrow \nabla \tilde{u}$ uniformly in $C^0(\Omega)$.

Next, we show $\tilde{u}|_K$ is a solution to problem \eqref{Eq:patr} on $K$. By the Hausdorff convergence $d_H(K_i, K)=\epsilon_i \to 0$, for any test function $\phi \in C_0^\infty(K)$, let
$\delta = \min_{x \in \operatorname{supp}(\phi)} \operatorname{dist}(x, \partial K) > 0$, there exists $i_0$ such that $\epsilon_i < \delta/2$ for all $i > i_0$. Thus for any $x \in \operatorname{supp}(\phi)$ and $y \in \partial K_i$, there exists $z \in \partial K$ such that $|y - z| \leq d_H(K_i, K) < \delta/2$ and 
$$
|x - y| \geq |x - z| - |z - y| > \delta - \frac{\delta}{2} = \frac{\delta}{2},
$$
thus $\operatorname{supp}(\phi) \subset K_i$, i.e., $\phi \in C_0^\infty(K_i)$.

Since $u_i$ is a solution to problem \eqref{Eq:patr} on $K_i$, then
\begin{align}\label{Eq:rjcz}
\int_{K_i} F^{p-1}(\nabla u_i) \langle \nabla_\xi F(\nabla u_i), \nabla \phi \rangle  dx = \int_{K_i} -\phi  dx.
\end{align}
Since $\tilde{u}_i = u_i$ on $K_i$, thus (\ref{Eq:rjcz}) can be rewritten as
\begin{align}\label{Eq:rjcz1}
\int_{K_i} F^{p-1}(\nabla \tilde{u}_i) \langle \nabla_\xi F(\nabla \tilde{u}_i), \nabla \phi \rangle  dx = \int_{K_i} -\phi  dx.
\end{align}
Since $\tilde{u}_i \to \tilde{u}$ and $\nabla \tilde{u}_i \to \nabla \tilde{u}$ uniformly on the compact set $\Omega$, and  $\operatorname{supp}(\phi)$ is a compact subset of $\Omega$, thus  $\tilde{u}_i \to \tilde{u}$ and $\nabla \tilde{u}_i \to \nabla \tilde{u}$ converge uniformly on $\operatorname{supp}(\phi)$, from (\ref{Eq:rjcz1}), we obtain
$$
\int_K F^{p-1}(\nabla \tilde{u}) \langle \nabla_\xi F(\nabla \tilde{u}), \nabla \phi \rangle  dx = \int_K -\phi  dx,
$$
this shows that $\tilde{u}|_K$ satisfies
$$
\Delta_p^F \tilde{u} = -1 \quad \text{in } K,
$$
since $\tilde{u}_i = 0$ on $\partial K_i$ and $\partial K_i \to \partial K$, the uniform convergence $\tilde{u}_i \to \tilde{u}$ implies
$$
\tilde{u} = 0 \quad \text{on } \partial K,
$$
thus $\tilde{u}|_K$ is the solution to (\ref{Eq:patr}). Due to the uniqueness of the solution to (\ref{Eq:patr}), we have
$$
\tilde{u}|_K = u.
$$
The gradient convergence $\nabla \tilde{u}_i \to \nabla \tilde{u}$ then implies 
$$
\nabla \tilde{u}|_K = \nabla u.
$$

Finally, we prove $u_i \to u$ uniformly in $C^1(\bar{K_i})$. For $x \in \bar{K_i}$, we obtain
$$
|u_i(x) - u(x)| \leq |\tilde{u}_i(x) - \tilde{u}(x)| + |\tilde{u}(x) - u(x)|.
$$
Since $|\tilde{u}_i(x) - \tilde{u}(x)| \leq \|\tilde{u}_i - \tilde{u}\|_{C^0(\Omega)} \to 0$ uniformly. For $x \in \bar{K_i} \setminus \bar{K}$,  by (\ref{Eq:Hsdc}),  there exists $y \in \partial K$ such that $|x - y| \leq \epsilon_i \to 0$. Moreover $\tilde{u} = u$ on $\bar{K}$ and $u \in C^0(\bar{K})$, we get
$$
|\tilde{u}(x) - u(x)| = |\tilde{u}(x) - \tilde{u}(y)| \leq \|\tilde{u}\|_{C^1(\Omega)} |x - y| \leq C \epsilon_i\to 0.
$$
Thus, 
\begin{align}\label{Eq:dopa}
\sup_{x \in \bar{K_i}} |\tilde{u}(x) - u(x)| \to 0 \text{ as } i \rightarrow \infty.
\end{align}

On the other hand, for $x \in \bar{K_i}$, we have
$$
|\nabla u_i(x) - \nabla u(x)| \leq |\nabla \tilde{u}_i(x) - \nabla \tilde{u}(x)| + |\nabla \tilde{u}(x) - \nabla u(x)|.
$$
Since $|\nabla \tilde{u}_i(x) - \nabla \tilde{u}(x)| \leq \|\nabla \tilde{u}_i - \nabla \tilde{u}\|_{C^0(\Omega)} \to 0$ uniformly. For $x \in \bar{K_i} \setminus \bar{K}$,  by (\ref{Eq:Hsdc}),  choose $y \in \partial K$ with $|x - y| \leq \epsilon_i\to0$. In addition $\nabla \tilde{u} = \nabla u$ on $\bar{K}$ and $\nabla \tilde{u} \in C^0(\Omega)$, we obtain
$$
|\nabla \tilde{u}(x) - \nabla u(x)| = |\nabla \tilde{u}(x) - \nabla \tilde{u}(y)| \leq \|\nabla \tilde{u}\|_{C^{0,1}(\Omega)} |x - y| \leq C \epsilon_i\to0.
$$
Thus, 
\begin{align}\label{Eq:dopa1}
\sup_{x \in \bar{K_i}} |\nabla \tilde{u}(x) - \nabla u(x)| \to 0 \text{ as } i \rightarrow \infty.
\end{align}
From (\ref{Eq:dopa}) and (\ref{Eq:dopa1}), we deduce
$$
\|u_i - u\|_{C^1(\bar{K_i})} = \sup_{x \in \bar{K_i}} |u_i(x) - u(x)| + \sup_{x \in \bar{K_i}} |\nabla u_i(x) - \nabla u(x)| \to 0 \text{ as } i \rightarrow \infty.
$$
Therefore, $u_i \to u$ uniformly in $C^1(\bar{K_i})$.
\end{proof}

Next, we will prove that the anisotropic $p$-torsional measure is weakly convergent.

\begin{lemma}\label{lemapytrs}
 Let $1<p<\infty$ and $K,K_i \subset \mathbb{R}^n$ be bounded domain with the boundary of class $C^{2, \alpha}$ and $F$ be a norm in $\mathscr{I}_p$. If $K_i$ converges to $K$ in Hausdorff metric as $i \rightarrow \infty$. Then the sequence of measures $S_{F,p}\left(K_i, \cdot\right)$ weakly converges to $S_{F,p}(K, \cdot)$.
\end{lemma}
\begin{proof}
Since $K_i, K\subset \mathbb{R}^n$ is a bounded domain with the boundary of class $C^{2, \alpha}$ and $F$ is a norm in $\mathscr{I}_p, p>1$. By Lemma \ref{lemia}, there exists a unique solution $u$ to Problem (\ref{Eq:patr}) in $K$ with $u \in C^{1, \alpha}(\bar{K})$. Correspondingly, in $K_i$ there exists a unique solution $u_i \in C^{1, \alpha}(\bar{K_i})$. In addition, there exists $0<c<1$ such that
\begin{align}\label{Eq:tdguJ}
c \leq |\nabla u| \leq \frac{1}{c} \ \text{on} \ \partial K, \quad c \leq |\nabla u_i| \leq \frac{1}{c} \ \text{on} \ \partial K_i,
\end{align}

In order to prove that $S_{F,p}(K_i,\cdot)$ weakly converges to $S_{F,p}(K,\cdot)$, it is sufficient to prove that for any continuous function $f \in C(\mathbb{S}^{n-1})$ such that
$$
\int_{\mathbb{S}^{n-1}} f(\eta) d S_{F,p}(K_i, \eta) \rightarrow \int_{\mathbb{S}^{n-1}} f(\eta) d S_{F,p}(K, \eta),
$$
this is equivalent to prove that
\begin{align}\label{Eq:wwa}
\int_{\partial K_i} f\left(\mathbf{g}_i(x)\right)F^p(\nabla u_i(x))d \mathscr{H}^{n-1}(x) \rightarrow \int_{\partial K} f(\mathbf{g}(x))F^p(\nabla u(x)) d \mathscr{H}^{n-1}(x).
\end{align}

Note that,
\begin{align}\label{Eq:kdt9}
\nonumber& \bigg|\int_{\partial K_i} f\left(\mathbf{g}_i(x)\right)F^p(\nabla u_i(x))d \mathscr{H}^{n-1}(x) -\int_{\partial K} f(\mathbf{g}(x))F^p(\nabla u(x)) d \mathscr{H}^{n-1}(x)\bigg| \\
\nonumber\leq& \int_{\partial K_i}\left|f(\mathbf{g}_i(x))\right| \cdot\bigg|F^p(\nabla u_i(x))-F^p(\nabla u(x))\bigg| d \mathscr{H}^{n-1}(x) \\
&+\bigg|\int_{\partial K_i} f(\mathbf{g}_i(x))F^p(\nabla u(x))d \mathscr{H}^{n-1}(x)-\int_{\partial K} f(\mathbf{g}(x))F^p(\nabla u(x))d \mathscr{H}^{n-1}(x)\bigg|.
\end{align}
where $\mathbf{g}_i, \mathbf{g}$ be Gauss maps of $K_i, K$ respectively.

We firstly prove that
\begin{align}\label{Eq:snxzi}
\int_{\partial K_i}\left|f(\mathbf{g}_i(x))\right| \cdot\bigg|F^p(\nabla u_i(x))-F^p(\nabla u(x))\bigg| d \mathscr{H}^{n-1}(x) \rightarrow 0, \text { as } i \rightarrow \infty.
\end{align}

By the continuity of $f$ on $\mathbb{S}^{n-1}$,  Lemma \ref{lemhdy} and  (\ref{Eq:Fyjd}) and (\ref{Eq:tdguJ}), use (\ref{Eq:Fyjd}) again, we have
\begin{align}\label{Eq:snxz2}
\nonumber&\int_{\partial K_i}\left|f(\mathbf{g}_i(x))\right| \cdot\bigg|F^p(\nabla u_i(x))-F^p(\nabla u(x))\bigg| d \mathscr{H}^{n-1}(x)\\
\nonumber\leq&C \int_{\partial K_i} \bigg|F^p(\nabla u_i(x))-F^p(\nabla u(x))\bigg| d \mathscr{H}^{n-1}(x)\\
\nonumber\leq&C \int_{\partial K_i} \bigg|p[F(\nabla u_i(x))+F(\nabla u(x))]^{p-1} F(\nabla u_i(x)-\nabla u(x))\bigg| d \mathscr{H}^{n-1}(x)\\
\nonumber\leq& C_1\left(\int_{\partial K_i}\bigg| F(\nabla u_i(x)-\nabla u(x)) \bigg|d \mathscr{H}^{n-1}(x)\right)\\
\leq& C_2\left(\int_{\partial K_i}\bigg| \nabla u_i(x)-\nabla u(x) \bigg|d \mathscr{H}^{n-1}(x)\right),
\end{align}
where $C, C_1, C_2>0$ are independent of $i$. 

Since $ \partial K_i$ is $ C^{2,\alpha}$, from Lemma \ref{lemapytrs}, since $u_i \to u$ uniformly in $C^1(\bar{K_i})$, it implies that $\sup_{x\in \partial K_i}| \nabla u_i - \nabla u |\to 0$, we obtain 
\begin{align}\label{Eq:gjcl}
\int_{\partial K_i}| \nabla u_i(x)-\nabla u(x) |d \mathscr{H}^{n-1}(x) \leq \sup_{x\in \partial K_i}| \nabla u_i - \nabla u |\mathscr{H}^{n-1}(\partial K_i)\to0,
\end{align}  
from this and (\ref{Eq:gjcl}) and (\ref{Eq:snxz2}), we obtain (\ref{Eq:snxzi}).

Now, we deal with the second summand of the right hand of (\ref{Eq:kdt9}). Let
$$
\mathscr{I}(x)=f(\mathbf{g}(x))F^p(\nabla u(x)), \quad \mathscr{I}_i(x)=f(\mathbf{g}_i(x))F^p(\nabla u(x)), i \in \mathbb{N}.
$$
 Combine the continuity of $f$ on $\mathbb{S}^{n-1}$, by (\ref{Eq:tdguJ}) and (\ref{Eq:Fyjd}),  there exists a constant $C_3>0$ such that
$$
\|\mathscr{I}(x)\|_{L^{\infty}(\partial K)}\leq C_3, \quad \quad \left\|\mathscr{I}_i\right\|_{L^{\infty}\left(\partial K_i\right)} \leq C_3, \quad \forall i \in \mathbb{N}.
$$

Let $x \in \partial K$ and $x_i \in \partial K_i$ such that $x_i \rightarrow x$ non-tangentially. Since $\mathbf{g}\left(x_i\right) \rightarrow \mathbf{g}(x)$, then $f\left(\mathbf{g}_i(x)\right) \rightarrow f(\mathbf{g}(x))$
 and $\nabla u(x_i)\rightarrow\nabla u(x)$. These imply that $g_i \rightarrow g$ as $i \rightarrow \infty$. Thus, by Lemma \ref{lemfqjix} and \cite[Lemma 3.4]{CAM2010}, we have
$$
\bigg|\int_{\partial K_i} f(\mathbf{g}_i(x))F^p(\nabla u(x))d \mathscr{H}^{n-1}(x)-\int_{\partial K} f(\mathbf{g}(x))F^p(\nabla u(x))d \mathscr{H}^{n-1}(x)\bigg|\rightarrow0, \text { as } i \rightarrow \infty.
$$
From this and together with (\ref{Eq:kdt9}) and (\ref{Eq:snxzi}), we obtain (\ref{Eq:wwa}). That completes the proof.
\end{proof}

From Lemma \ref{lemapytrs}, (\ref{Eq:AFNZg}) and the continuity of the support function, we can easily deduce that the anisotropic $p$-torsional rigidity $\tau_{F,p}$ is continuous.

\begin{lemma}\label{lemlx1t}
 Let $1<p<\infty$ and $K,K_i \subset \mathbb{R}^n$ be a bounded domain with the boundary of class $C^{2, \alpha}$ and $F$ be a norm in $\mathscr{I}_p$. If $K_i$ converges to $K$ in Hausdorff metric as $i \rightarrow \infty$. Then $\tau_{F,p}$ is continuous, that is,
$$
\tau_{F,p}(K_i) \rightarrow \tau_{F,p}(K).
$$
\end{lemma}

Finally, we will derive the variational formula for the anisotropic $p$-torsional rigidity.

\begin{lemma}\label{lembfgsn}
Let $1<p<\infty$ and $K$ be a convex body containing the origin in its interior, such that $\partial K$ up to set of $(n-1)$-dimensional Hausdorff measure zero, and $f: \mathbb{S}^{n-1} \rightarrow \mathbb{R}$ be a continuous function. For sufficiently small $\delta>0$ and each $t \in(-\delta, \delta)$, the continuous function $h_t: \mathbb{S}^{n-1} \rightarrow(0, \infty)$ is defined by
$$h_t(v)=h_K(v)+t f(v), \ \ \ v \in \mathbb{S}^{n-1},$$
then
$$
\lim _{t \rightarrow 0} \frac{\tau_{F,p}\left([h_t]\right)-\tau_{F,p}(K)}{t}=\int_{\mathbb{S}^{n-1}} f(v) d S_{F,p}(K,v).
$$
\end{lemma}
\begin{proof}
Begin by writing in polar coordinates, with $h_t=h_K+t f$,
$$
\tau_{F,p}\left(\left[h_t\right]\right)=\int_K F^p(\nabla u(x)) d x=\int_{\mathbb{S}^{n-1}} \int_0^{\rho_{\left[h_t\right]}(v)} F^p(\nabla u(\xi v)) \xi^{n-1} d \xi d u=\int_{\mathbb{S}^{n-1}} G_t(v) d v,
$$
where $G_t(v)=\int_0^{\rho_{\left[h_t\right]}(v)}F^p(\nabla u(\xi v)) \xi^{n-1} d \xi$. By the Lebesgue differentiation theorem, we know that 
$\left.\frac{\mathrm{d} G_t}{\mathrm{~d} t}\right|_{t=0}=G_0^{\prime}$ exists. Let $\mathcal{I}=\left[\rho_K(v), \rho_{\left[h_t\right]}(v)\right]$, for all $v \in \mathbb{S}^{n-1}$, note $F^p(\nabla u(\xi v))$ exists from the assumption on $F$, by Lemma \ref{lem:jxy}, we obtain
\begin{align*}
\lim _{t \rightarrow 0} \frac{G_t(v)-G_0(v)}{t} & =\lim _{t \rightarrow 0} \frac{1}{t} \int_\mathcal{I} F^p(\nabla u(\xi v)) \xi^{n-1} d \xi \\
& =\lim _{t \rightarrow 0} \frac{\rho_{\left[h_t\right]}(v)-\rho_K(v)}{t} \frac{1}{|\mathcal{I}|} \int_\mathcal{I} F^p(\nabla u(\xi v)) \xi^{n-1} d \xi \\
& =F^p(\nabla u(r_K(v))) \frac{f\left(\mathbf{g}_K\left(r_K(v)\right)\right)}{h_K\left(\mathbf{g}_K\left(r_K(v)\right)\right)} \rho_K^n(v).
\end{align*}
Meanwhile, for some $M>0$,
$$
\frac{f\left(\mathbf{g}_K\left(r_K(v)\right)\right)}{h_K\left(\mathbf{g}_K\left(r_K(v)\right)\right)} \rho_K(v)=\lim _{t \rightarrow 0} \frac{\rho_{\left[h_t\right]}(v)-\rho_K(v)}{t} \leq M,
$$
from the Lemma \ref{lem:jxy}, we have that the derivative is dominated by the integrable function 
$$M\left(\max _{u \in \mathbb{S}^{n-1}} \rho_K(v)\right)^{n-1} F^p(\nabla u(r_K(v))).$$
By dominated convergence to differentiate underneath the integral sign and combining Lemma \ref{lem:jxy1} and (\ref{Eq:gdfs}), we derive
\begin{align*}
\int_{\mathbb{S}^{n-1}} G_0^{\prime}(v) d u & =\int_{\mathbb{S}^{n-1}} F^p(\nabla u(r_K(v))) \frac{f\left(\mathbf{g}_K\left(r_K(v)\right)\right)}{h_K\left(\mathbf{g}_K\left(r_K(v)\right)\right)} \rho_K^n(v) d v \\
& =\int_{\partial K} F^p(\nabla u(x)) f\left(\mathbf{g}_K(x)\right) d x\\
& =\int_{\mathbb{S}^{n-1}} f(v) d S_{F,p}(K,v).
\end{align*}
This completes the proof of Lemma \ref{lembfgsn}.
\end{proof}

Let
$$
h_t(v)=h_K(v)+tf(v) h_K(v), \quad v \in \mathbb{S}^{n-1},
$$
in Lemma \ref{lembfgsn}, we can immediately get the following result.

\begin{lemma}\label{lemlogb}
Let $1<p<\infty$ and $K$ be a convex body containing the origin in its interior, such that $\partial K$ up to set of $(n-1)$-dimensional Hausdorff measure zero, and $f: \mathbb{S}^{n-1} \rightarrow \mathbb{R}$ be continuous function. For sufficiently small $\delta>0$ and each $t \in(-\delta, \delta)$, a continuou function $h_t: \mathbb{S}^{n-1} \rightarrow(0, \infty)$ is defined by
$${\rm{log}} h_t(v)={\rm{log}} h_K(v)+t f(v), \ \ \ v \in \mathbb{S}^{n-1},$$
then
$$
\lim _{t \rightarrow 0} \frac{\tau_{F,p}\left([h_t]\right)-\tau_{F,p}(K)}{t}=\frac{n(p-1)+p}{p-1}\int_{\mathbb{S}^{n-1}} f(v) d \tau_{F,p}^{\rm{log}}(K,v).
$$
\end{lemma}

\section{A variational proof of Minkowski problem to anisotropic $p$-torsional measure}\label{SEC4}

In this section, we address the Minkowski problem associated with the anisotropic $p$-torsional measure $S_{F,p}(K, \cdot)$. 

The subsequent lemma demonstrates that the centroid of every anisotropic $p$-torsional measure $S_{F,p}(K, \cdot)$ coincides with the origin. This property is recognized as a necessary condition for a measure to be considered the anisotropic $p$-torsional measure of a convex body. 

\begin{lemma}\label{lembyxzm}
Let $1<p<\infty$ and $K$ be a convex body containing the origin in its interior, the anisotropic $p$-torsional measure $S_{F,p}(K, \cdot)$ satisfies
$$
\int_{\mathbb{S}^{n-1}} v d S_{F,p}(K, v)=0.
$$
\end{lemma}
\begin{proof}
Let $u \in \mathbb{S}^{n-1}$ and $\lambda >0$, such that $K+\lambda u$ be a convex body containing the origin in its interior. From (\ref{Eq:AFNZg}) and ({\bf c}) and ({\bf b}) of Proposition \ref{prop:apxz}, we have
\begin{align*}
\int_{\mathbb{S}^{n-1}} \lambda u \cdot v d S_{F,p}(K, v) & =\int_{\mathbb{S}^{n-1}}\left(h(K,v)+\lambda u \cdot v-h(K,v)\right) d S_{F,p}(K, v)\\
& =\int_{\mathbb{S}^{n-1}} h(K+\lambda u, v) d S_{F,p}(K+\lambda u, v)-\int_{\mathbb{S}^{n-1}} h(K,v) d S_{F,p}(K, v) \\
& =\frac{n(p-1)+p}{p-1}\tau_{F,p}(K+\lambda u )-\frac{n(p-1)+p}{p-1}\tau_{F,p}\\
& =0.
\end{align*}
Since $u \in \mathbb{S}^{n-1}$ is arbitrary and $\lambda$ is positive, we obtain
$$
\int_{\mathbb{S}^{n-1}} v dS_{F,p}(K, v)=0.
$$
This completes the proof.
\end{proof}

\begin{lemma}\label{lembyxzm1}
Suppose $1<p<\infty$. If $\mu$ is a non-zero, finite Borel measure on $\mathbb{S}^{n-1}$, then there exists $a$ convex body $K \in \mathscr{K}^n$ so that
$$
S_{F,p}(K,\cdot)=\mu,
$$
then $\mu$ is not concentrated on any closed hemisphere of $\mathrm{S}^{n-1}$, and
\begin{align}\label{Eq:Cxbtj}
\int_{\mathbb{S}^{n-1}} v d \mu(v)=0.
\end{align}
\end{lemma}
\begin{proof}
Since the surface area measure of a convex body containing the origin in its interior cannot be concentrated in a closed hemisphere of $\mathbb{S}^{n-1}$ (see \cite[Proposition 7.3.1]{BMT2024}), from (\ref{Eq:AMBJCD}), we knows that the anisotropic $p$-torsional measure, $S_{F,p}(K,\cdot)$, of $K$ is absolutely continuous with respect to the surface area measure $S(K,\cdot)$ of $K$. Thus the anisotropic $p$-torsional measure $S_{F,p}(K,\cdot)$, of $K$ cannot be concentrated in a closed hemisphere of $\mathbb{S}^{n-1}$. Combined with Lemma \ref{lembyxzm}, by $S_{F,p}(K,\cdot)=\mu$,
then $\mu$ is not concentrated on any closed hemisphere of $\mathrm{S}^{n-1}$ and $\int_{\mathbb{S}^{n-1}} v d \mu(v)=0$.
\end{proof}

For a finite Borel measure $\mu$ on $\mathbb{S}^{n-1}$, and $g \in C^{+}(\mathbb{S}^{n-1})$, the positive continuous function on $\mathbb{S}^{n-1}$, let
$$
\|g: \mu\|=\int_{\mathbb{S}^{n-1}} g(v) d \mu(v).
$$

Define the functional $\Psi_{F,p, \mu}: \mathscr{K}_o^n \rightarrow \mathbb{R}$, by
$$
\Psi_{F,p, \mu}(K)=\frac{\|h_K: \mu\|}{\tau_{F,p}\left(K\right)^{\frac{p-1}{np+p-n}}}.
$$

The following lemma reduces the Minkowski problem for the anisotropic $p$-torsional measure to be a minimization of a functional on convex bodies.

\begin{lemma}\label{lemzyhwt}
 Suppose $1<p<\infty$, and $\mu$ is a nonzero finite Borel measure on $\mathbb{S}^{n-1}$. If the minimization problem
$$
\inf \left\{ \Psi_{F,p, \mu}(K): K \in \mathscr{K}_o^n\right\},
$$
has a solution in $\mathscr{K}_o^n$, then there exists a body $K \in \mathscr{K}_o^n$ so that
$$
S_{F,p}(K,\cdot)=\mu.
$$
\end{lemma}

\begin{proof}
By assumption, the minimization problem has a solution $K_0 \in \mathscr{K}_o^n$. The origin being in the interior of $K_0$ implies $h_{K_0}>0$.

Suppose $f \in C^+(\mathbb{S}^{n-1})$. Then $h_t=h_{K_0}+tf>0$, whenever $|t|<\delta$, for some sufficiently small $\delta>0$. Obviously, the function $t \mapsto \Psi_{F,p, \mu}(K_t)$ attains its minimum at $t=0$. Note that $t \mapsto \Psi_{F,p, \mu}(K_t)$ may not be differentiable at $t=0$ (because $t \mapsto\left\|h_{K_t}: \mu\right\|$ may not be differentiable at $t=0$ ). Let
$$
\psi(t)=\Psi_{F,p, \mu}(K_t)=\frac{\|h_t: \mu\|}{\tau_{F,p}\left(K_t\right)^{\frac{p-1}{np+p-n}}}.
$$

Since the body $K_t$ is the Wulff shape of the function $h_t$ and the measure $\mu$ is non-negative it follows that $\left\|h_{K_t}: \mu\right\| \leq\left\|h_t: \mu\right\|$, and 
$$
\psi(t) = \Psi_{F,p, \mu}(K_t) \geq \Psi_{F,p, \mu}(K_0)=\psi(0).
$$
Thus, $\psi(t)$ attains minimum at $t=0$. By Lemma \ref{lembfgsn}, the function $t \mapsto \tau_{F,p}\left(K_t\right)$ is differentiable at $t=0$. Of course, $t \mapsto\left\|h_t: \mu\right\|$ is obviously differentiable at $t=0$. Therefore,
$$
\begin{aligned}
0 & =\left.\frac{d}{d t} \psi(t)\right|_{t=0} \\
& =\frac{\int_{\mathbb{S}^{n-1}} f(v) d \mu(v)}{\tau_{F,p}\left(K_0\right)^{\frac{p-1}{np+p-n}}}-\frac{\int_{\mathbb{S}^{n-1}} h_{K_0} d \mu(v)\frac{p-1}{np+p-n}\tau_{F,p}\left(K_0\right)^{\frac{p-1}{np+p-n}-1}\frac{d}{d t}\tau_{F,p}(K_0)}{\bigg(\tau_{F,p}\left(K_0\right)^{\frac{p-1}{np+p-n}}\bigg)^2}\\
&=\frac{\int_{\mathbb{S}^{n-1}} f(v) d \mu(v)}{\tau_{F,p}\left(K_0\right)^{\frac{p-1}{np+p-n}}}-\frac{\int_{\mathbb{S}^{n-1}} h_{K_0} d \mu(v)\frac{p-1}{np+p-n}\tau_{F,p}\left(K_0\right)^{\frac{p-1}{np+p-n}-1}\int_{\mathbb{S}^{n-1}} f(v) d S_{F,p}(K_0,v)}{\bigg(\tau_{F,p}\left(K_0\right)^{\frac{p-1}{np+p-n}}\bigg)^2}.
\end{aligned}
$$
Since this hold for all $f\in C(\mathbb{S}^{n-1})$, we conclude that
$$
\frac{p-1}{np+p-n} \frac{\left\|h_{K_0}: \mu\right\|}{\tau_{F,p}(K_0)}S_{F,p}(K_0,\cdot)=\mu.
$$

Now, letting $K=\lambda K_0$ with $\lambda^{\frac{np-n+1}{p-1}}=\frac{p-1}{np+p-n} \frac{\left\|h_{K_0}: \mu\right\|}{\tau_{F,p}(K_0)}S_{F,p}(K_0,\cdot)$ gives the desired body.
\end{proof}

\begin{proof}[Proof of Theorem \ref{THm:jdMwt}]
 The necessity of the condition is established in Lemma \ref{lembyxzm1}. We proceed to establish sufficiency. We shall demonstrate that the minximization problem
$$
\inf \left\{ \Psi_{F,p, \mu}(K): K \in \mathscr{K}_o^n\right\},
$$
has a solution in $\mathscr{K}_o^n$. To this end, suppose $K_i$ is a minimizing sequence for the functional $\Psi_{F,p, \mu} : \mathscr{K}_o^n \rightarrow \mathbb{R}$.

Since the measure $\mu$ satisfies the condition (\ref{Eq:Cxbtj}), it follows that $K \mapsto\left\|h_K: \mu\right\|$ is a translation invariant function. Since $\tau_{F,p} (K)$ is translation invariant, we now see that $\Psi_{F,p, \mu}$ is translation invariant as well. We therefore may assume that the centroid of each of the $K_i$ is at the origin.

Since $\Psi_{F,p, \mu}$ is homogeneous of degree 0, we may rescale the bodies in our minimizing sequence so that $\tau_{F,p}\left(K_i\right)=1$, for all $i$. So, $K_i$ is a minimizing sequence for the problem,
$$
\inf \left\{\left\|h_K: \mu\right\|: \tau_{F,p}(K)=1, K \in \mathscr{K}_o^n\right\}.
$$

Now we show that the sequence of $\left\{K_i\right\}$ is bounded. Since $K_i$ is a minimization sequence, we can assume that the ball $r\mathbb{B}^{n}$ centered at the origin and whose radius $r$ is such that $\tau_{F,p}(r\mathbb{B}^{n})=1$, we have
$$
\int_{\mathbb{S}^{n-1}}h_{K_j}(v) d \mu(v) \leq\int_{\mathbb{S}^{n-1}} h_{r\mathbb{B}^{n}}(v) d \mu(v)=r\mu(\mathbb{S}^{n-1}).
$$
Let $R_{K_j}$ be the maximal radius of $K_j$. If $x_j\in \mathbb{S}^{n-1}$ is the direction of this radius, then
$R_{K_j}\left(x_j \cdot v\right)_{+} \leq h_{K_j}(v)$, for all $v \in \mathbb{S}^{n-1}$. We have
$$
\begin{aligned}
R_{K_j} \int_{\mathbb{S}^{n-1}}\left(x_j \cdot v\right)_{+} d \mu(v) & \leq \int_{\mathbb{S}^{n-1}} h_{K_j}(v) d \mu(v)=r\mu(\mathbb{S}^{n-1}).
\end{aligned}
$$

Since $\mu$ is not concentrated on a closed hemisphere, there exists a constant $a>0$, so that
$$
\int_{\mathbb{S}^{n-1}}(x \cdot v)_{+} d \mu(v) \geq a,
$$
for all $v \in \mathbb{S}^{n-1}$. Thus,
$$
R_{K_j} \leq \frac{r\mu(\mathbb{S}^{n-1})}{a}.
$$
 Hence the sequence $K_i$ is bounded. By the Blaschke selection theorem \cite[Theorem 1.8.7]{SRA2014}, there exists a subsequence, which we again denote as $K_i$, that converges to a compact convex set $K_0$. The continuity of $\tau_{F,p}$ assures that $\tau_{F,p}\left(K_0\right)=1$.

Since $u$ be the solution to (\ref{Eq:patr}), by (\ref{Eq:AFNZg}) and (\ref{Eq:Fyjd}), there is a constant $b>0$ such that
\begin{align*}
1=&\tau_{F,p}\left(K_0\right)\\
=&\frac{p-1}{n(p-1)+p} \int_{\mathbb{S}^{n-1}} h_{K_0}(\xi) F^p(\nabla u(\mathbf{g}^{-1}_{K_0}(\xi)))d S(K_0,\xi) \\
\leq&b^p\frac{p-1}{n(p-1)+p} \int_{\mathbb{S}^{n-1}} h_{K_0}(\xi) |\nabla u(\mathbf{g}^{-1}_{K_0}(\xi))|^pd S(K_0,\xi).
\end{align*}
On the other hand, because of $K_0 \subset \mathbb{R}^n$ be a bounded domain with the boundary of class $C^{2, \alpha}$, by  Lemma \ref{lemia} there exists a unique solution $u$ to Problem (\ref{Eq:patr}) with $u \in C^{1, \alpha}(K_0)$ and  $|\nabla u(\mathbf{g}^{-1}_{K_0}(\xi))|$ is bounded, there is a constant $c_1>0$ such that
$$
1 \leq (bc)^p \frac{n(p-1)}{n(p-1)+p}|K_0| \Rightarrow |K_0| \geq \frac{1}{(bc_1)^p}\frac{n(p-1)+p}{n(p-1)}>0.
$$
Thus, $K_0 \in \mathscr{K}_o^n$. Since $K_i$ is a minimizing sequence for $\Psi_{F,p, \mu}$, the body $K_0$ is a minimizer of $\Psi_{F,p, \mu}$, and Lemma \ref{lemzyhwt} now gives the desired result.

\end{proof}

\section{The log Minkowski problem of anisotropic $p$-torsional measure}\label{SEC5}

In this section, inspired by the approaches of Zhu \cite{ZGX2014}, Guo, Xi and Zhao \cite{GLJ2024}, and Hu \cite{HJR20241} to the log-Minkowski problem, this section considers the log-Minkowski problem of anisotropic $p$-torsional rigidity. First, we prove the existence of solutions to the log-Minkowski problem of anisotropic $p$-torsional rigidity in the case of discrete measures whose supports are in general position.

\subsection{The discrete case}

Let $N \geq n+1$, and let $\{u_1, \dots, u_N\} \subset \mathbb{S}^{n-1}$ be in general position and not concentrated on a closed hemisphere. Given positive real numbers $\alpha_1, \dots, \alpha_N$, define for $P \in \mathcal{P}(u_1, \dots, u_N)$ and $\mu = \sum_{i=1}^N \alpha_i \delta_{u_i}$ the functional
\begin{align}\label{Eq:Fdhs}
\Psi_{\mu, P}(\eta)=\sum_{i=1}^N \alpha_i \log \left(h_P\left(u_i\right)-\eta \cdot u_i\right).
\end{align}

Note that an extremal problem of the functional $\Psi_{\mu, P}$ subject to a volume constraint has served as a common approach for formulating the logarithmic Minkowski problem (see \cite{ZGX2014}).  However, in this paper we consider the extreme problem
\begin{align}\label{eq:IMEP}
\inf \left\{\max _{\eta \in \operatorname{Int}(K)} \Psi_{\mu, K}(\eta): K \in \mathcal{P}_N\left(u_1, \ldots, u_N\right) \text { and } \tau_{F,p}(K)=1\right\}.
\end{align} 

Below, we will show that the minimizer of the problem (\ref{eq:IMEP}) solves the discrete log Minkowski problem for anisotropic $p$-torsional measure. Before proving this, we first introduce some known conclusions about the function $\Psi_{\mu, P}$.

\begin{lemma}[\cite{ZGX2014,LEE2021}]\label{lem:zgx}
Suppose that $\left\{u_1, \ldots, u_N\right\} \subset \mathbb{S}^{n-1}$ is in general position and not concentrated on a closed hemisphere. Let $\alpha_1, \ldots, \alpha_N$ be $N$ positive real numbers and define $\mu=\sum_{i=1}^N \alpha_i \delta_{u_i}$. If $P \in \mathcal{P}\left(u_1, \ldots, u_N\right)$, then there exists a unique point $\eta(P) \in \operatorname{Int}(P)$ such that
\begin{align}\label{eq:zgx1}
\Psi_{\mu, P}(\eta(P))=\max _{\eta \in \operatorname{Int}(P)} \Psi_{\mu, P}(\eta).
\end{align}
Moreover, if a sequence of polytopes $P_i \in \mathcal{P}\left(u_1, \ldots, u_N\right)$ converges to $P$ with respect to the Hausdorff metric, then
$$
\lim _{i \rightarrow \infty} \eta\left(P_i\right)=\eta(P),
$$
and
$$
\lim _{i \rightarrow \infty} \Psi_{\mu, P_i}\left(\eta\left(P_i\right)\right)=\Psi_{\mu, P}(\eta(P)).
$$
\end{lemma}

Let us now prove that a minimizer of the extreme problem (\ref{eq:IMEP}) solves the discrete log Minkowski problem for anisotropic $p$-torsional measure.

\begin{lemma}[\cite{GLJ2024}]\label{lempwj}
Let $u_1, \ldots, u_N$ be $N$ unit vectors that are not contained in any closed hemisphere, and $P_i$ be a sequence of polytopes in $\mathcal{P}(u_1, \ldots, u_N)$. Assume the vectors $u_1, \ldots, u_N$ are in general position in dimension $n$. If outer radius $R_i$ of $P_i$ is not uniformly bounded in $i$, then its inner radius $r_i$ is not uniformly bounded in $i$ either.
\end{lemma}

\begin{corollary}\label{corpwj} Let $u_1, \ldots, u_N \in \mathbb{S}^{n-1}$ be $N$ unit vectors that are not contained in any closed hemisphere and $P \in \mathcal{P}(u_1, \ldots, u_N)$. Assume that $u_1, \ldots, u_N$ are in general position in dimension $n$. If the outer radius $R_i$ of $P_i$ is not uniformly bounded, then the anisotropic $p$-torsional rigidity $\tau_{F,p}(P_i)$ is also unbounded.
\end{corollary}
\begin{proof}
By Lemma \ref{lemxjgj}, we obtain that the anisotropic $p$-torsional rigidity $\tau_{F,p}(\mathbb{B}^{n})$ is positive for the centered unit ball $\mathbb{B}^{n}$. From this result, together with Lemma \ref{lempwj} and the homogeneity and translation invariance properties of the anisotropic $p$-torsional rigidity $\tau_{F,p}$, we can immediately derive this corollary.
\end{proof}

\begin{lemma}\label{lemqpa6c}
 Suppose $1<p<\infty$. If $\alpha_1, \ldots, \alpha_N \in(0, \infty)$, the unit vectors $u_1, \ldots, u_N$$(N \geq n+1)$ are in general position in dimension $n$, then there exists a polytope $P \in \mathcal{P}\left(u_1, \ldots, u_N\right)$ solving (\ref{eq:IMEP}) such that $P$ has exactly $N$ facets, $\eta(P)=o$, $\tau_{F,p}(P)=1$ and
$$
\Psi_{\mu,P}(o)=\inf \left\{\max _{\eta\in K} \Psi_{\mu,K}(\eta):K\in \mathcal{P}\left(u_1, \ldots, u_N\right), \tau_{F,p}(K)=1\right\}.
$$
\end{lemma}
\begin{proof}
Due to the translation invariance of $\Psi_{\mu,P}$, we can select a sequence $P_i \in \mathcal{P}\left(u_1, \ldots, u_N\right)$ such that $\eta\left(P_i\right)=o$ and $\tau_{F,p}(P_i)=1$, This sequence $\{P_i\}$ is a minimizing sequence for problem (\ref{eq:IMEP}).

Corollary \ref{corpwj} implies that the sequence $\{P_i\}$ is bounded. Then by the Blaschke selection theorem \cite[Theorem 1.8.7]{SRA2014}, there exists a subsequence $P_{i_j}$ of $P_i$, such that $P_{i_j} \rightarrow P$ in the Hausdorff metric. By Lemma \ref{lemlx1t}, we have $\tau_{F,p}(P_{i_j} ) \rightarrow \tau_{F,p}(P)$. From (\ref{Eq:sjgj}), we have 
\begin{align*}
1=\tau_{F,p}(P) \leq\frac{p-1}{n(p-1)+p}n^{-\frac{1}{p-1}}\kappa_n^{-\frac{p}{n(p-1)}}|P|^{\frac{n(p-1)+p}{n(p-1)}}.
\end{align*}
Consequently, we have
\begin{align*} 
|P|\geq\left(\frac{n(p-1)+p}{p-1}n^{\frac{1}{p-1}}\kappa_n^{\frac{p}{n(p-1)}}\right)^{\frac{n(p-1)}{n(p-1)+p}}>0.
\end{align*}
Thus, $P$ is non-degenerate.

Now, by Lemma \ref{lem:zgx}, we have $\eta(P)=\lim _{i \rightarrow \infty} \eta\left(P_i\right)=o$. By the definition of $\Psi_{\mu,P}$, we get
$$
\Psi_{\mu,P}(o)=\lim _{i \rightarrow \infty} \Psi_{\mu,P_i}(o)=\inf \left\{\max _{\eta \in K} \Psi_{\mu,K}(\eta): K\in \mathcal{P}\left(v_1, \ldots, v_N\right), \tau_{F,p}(K)=1\right\}.
$$

Next, we use the method of contradiction to prove that $P$ has exactly $N$ facets. Suppose there exists $i_0 \in\{1, \ldots, N\}$ such that $F\left(P, u_{i_0}\right)=P \cap H\left(P, u_{i_0}\right)$ is not the facet of $P$.

Since $P \in \mathcal{P}_N\left(u_1, \ldots, u_N\right)$, we can choose $|\varepsilon|$ small enough so that
$$
P_\varepsilon=P \cap\{\eta: \eta \cdot u_{i_0} \leq h(P, u_{i_0})-\eta\} \in \mathcal{P}(u_1, \ldots, u_N).
$$ 
Let
\begin{align}\label{Eq:lmdft}
\lambda(\varepsilon)=\tau_{F,p}\left(P_\varepsilon\right)^{-\frac{p-1}{n(p-1)+p}},
\end{align}
and define $\bar{P}_\varepsilon=\lambda(\varepsilon) P_\varepsilon$. Then $\bar{P}_\varepsilon \in \mathcal{P}_N\left(u_1, \ldots, u_N\right)$, $\tau_{F,p}(\bar{P}_\varepsilon)=1$, and $\bar{P}_\varepsilon \rightarrow P$ with respect to the Hausdorff metric as $\varepsilon \rightarrow 0$. 
Furthermore, by Lemma \ref{lem:zgx}, we know that
$$
\eta\left(P_\varepsilon\right) \rightarrow \eta(P)=o \in \operatorname{Int}(P) \text {, as } \varepsilon \rightarrow 0^{+}.
$$
Therefore, for a sufficiently small $\varepsilon>0$, we have
$$
\eta\left(P_\varepsilon\right) \in \operatorname{Int}(P),
$$
and
\begin{align}\label{Eq:hsexd}
h\left(P, u_k\right)>\eta\left(P_\varepsilon\right) \cdot u_k+\varepsilon, \text { for } k \in\{1, \ldots, N\}.
\end{align}
Now, let's prove 
\begin{align}\label{Eq:1hsexd}
\Psi_{\mu,\bar{P}_\varepsilon}\left(\eta\left(\bar{P}_\varepsilon\right)\right)<\Psi_{\mu,P}(\eta(P))=\Psi_{\mu,P}(o).
\end{align}
Since $\eta\left(\lambda(\varepsilon) P_\varepsilon\right)=\lambda(\varepsilon) \eta\left(P_\varepsilon\right)$, it can be deduced that 
\begin{align*}
& \Psi_{\mu,\bar{P}_\varepsilon}\left(\eta\left(\bar{P}_\varepsilon\right)\right) \\
=&\sum_{k=1}^N \alpha_k \log \left(h\left(\lambda(\varepsilon) P_\varepsilon, u_k\right)-\eta\left(\lambda(\varepsilon) P_\varepsilon\right) \cdot u_k\right) \\
=&\log \lambda(\varepsilon) \sum_{k=1}^N \alpha_k+\sum_{k=1}^N \alpha_k \log \left(h\left(P_\varepsilon, u_k\right)-\eta\left(P_\varepsilon\right) \cdot u_k\right) \\
=&\sum_{k=1}^N \alpha_k \log \left(h\left(P, u_k\right)-\eta\left(P_\varepsilon\right) \cdot u_k\right)-\alpha_{i_0} \log \left(h\left(P, u_{i_0}\right)-\eta\left(P_\varepsilon\right) \cdot u_{i_0}\right) \\
&+\log \lambda(\varepsilon) \sum_{k=1}^N \alpha_k +\alpha_{i_0} \log \left(h\left(P, u_{i_0}\right)-\eta\left(P_\varepsilon\right) \cdot u_{i_0}-\varepsilon\right) \\
=&\Psi_{\mu,P}\left(\eta\left(P_\varepsilon\right)\right)+G(\varepsilon),
\end{align*}
where
\begin{align}\label{Eq:Gdfs}
\nonumber G(\varepsilon)= & \log \lambda(\varepsilon)\sum_{k=1}^N \alpha_k-\alpha_{i_0} \log \left(h\left(P, u_{i_0}\right)-\eta\left(P_\varepsilon\right) \cdot u_{i_0}\right) \\
& +\alpha_{i_0} \log \left(h\left(P, u_{i_0}\right)-\eta\left(P_\varepsilon\right) \cdot u_{i_0}-\varepsilon\right).
\end{align}
To show that (\ref{Eq:1hsexd}) holds, we need to prove that $G(\varepsilon)<0$. Let $d_1$ be the diameter of $P$. By (\ref{Eq:hsexd}), we obtain
$$
0<h\left(P, u_{i_0}\right)-\eta\left(P_\varepsilon\right) \cdot u_{i_0}-\varepsilon<h\left(P, u_{i_0}\right)-\eta\left(P_\varepsilon\right) \cdot u_{i_0}<d_1.
$$
Based on the concavity of $\log x$ on the interval $[0, \infty)$, we can infer that
\begin{align*}
& \log\left(h(P, u_{i_0})-\eta\left(P_\varepsilon\right) \cdot u_{i_0}-\varepsilon\right)-\log \left(h(P, u_{i_0})-\eta\left(P_\varepsilon\right) \cdot u_{i_0}\right) \\
& \quad<\log \left(d_1-\varepsilon\right)-\log d_1.
\end{align*}
Therefore, combining (\ref{Eq:Gdfs}) and (\ref{Eq:lmdft}), we have
$$
G(\varepsilon)<H(\varepsilon),
$$
where
$$
H(\varepsilon)=-\frac{p-1}{n(p-1)+p} \log \tau_{F,p}\left(P_\varepsilon\right) \left(\sum_{k=1}^N \alpha_k\right)+\alpha_{i_0}\left(\log \left(d_1-\varepsilon\right)-\log d_1\right).
$$
Hence, by Lemma \ref{lembfgsn}, we obtain
\begin{align*}
H^{\prime}(\varepsilon) & =-\frac{p-1}{n(p-1)+p}\left(\sum_{k=1}^N \alpha_k\right) \frac{1}{\tau_{F,p}(P_\varepsilon)} \frac{d \tau_{F,p}(P_\varepsilon)}{d \varepsilon}-\frac{\alpha_{i_0}}{d_1-\varepsilon} \\
& =-\frac{p-1}{n(p-1)+p}\left(\sum_{k=1}^N \alpha_k\right) \frac{1}{\tau_{F,p}(P_\varepsilon)} \sum_{k=1}^N h^{\prime}(P_\varepsilon, v_k) dS_{F,p}(P_\varepsilon,\{v_k\})-\frac{\alpha_{i_0}}{d_1-\varepsilon}.
\end{align*}

Suppose $S_{F,p}(P_\varepsilon,\{v_k\})\neq 0$ for some $k \in\{1, \ldots, N\}$. By the absolute continuity of $S_{F,p}(P_\varepsilon,\cdot)$ with respect to $S(P, \cdot)$, we can infer that $S\left(P,\left\{v_k\right\}\right) \neq 0$. As a result, we can deduce that $P$ has a facet with normal vector $v_k$. By the definition of $P_\varepsilon$, for a sufficiently small $\varepsilon>0$, we have $h\left(P_\varepsilon, v_k\right)=h\left(P, v_k\right)$, this implies that $h^{\prime}\left(v_k, 0^{+}\right)=0$, where
$$
h^{\prime}(v_k, 0^{+})=\lim _{\varepsilon \rightarrow 0^{+}} \frac{h(P_\varepsilon, v_k)-h(P, v_k)}{\varepsilon}.
$$
From this and the fact that $P_\varepsilon \rightarrow P$ as $\varepsilon \rightarrow 0^{+}$, we obtain 
$$
\sum_{k=1}^N h^{\prime}(P_\varepsilon, v_k) S_{F,p}(P_\varepsilon,\{v_k\}) \rightarrow \sum_{k=1}^N h^{\prime}\left(v_k, 0^{+}\right) S_{F,p}(P,\{v_k\})=0 \text {, as } \varepsilon \rightarrow 0^{+},
$$
which indicates that, for a sufficiently small $\varepsilon$,
$H^{\prime}(\varepsilon)<0.$

Since $H(0)=0$, for sufficiently small $\varepsilon>0$, we can infer that $H(\varepsilon)<0$, which directly leads to $G(\varepsilon)<0$. Thus, there exists a $\varepsilon_0>0$ such that $P_{\varepsilon_0} \in \mathcal{P}\left(v_1, \ldots, v_N\right)$ and
$$
\Psi_{\mu,\bar{P}_{\varepsilon_0}}(\eta(\bar{P}_{\varepsilon_0}))<\Psi_{\mu,P}(\eta(P_{\varepsilon_0})) \leq \Psi_{\mu,P}(\eta(P))=\Psi_{\mu,P}(o).
$$
Let $P_0=\bar{P}_{\varepsilon_0}-\eta(\bar{P}_{\varepsilon_0})$, for $P_0 \in \mathcal{P}\left(v_1, \ldots, v_N\right)$, then we get
$$
\tau_{F,p}(P_0)=1, \qquad \eta\left(P_0\right)=o,\qquad \Psi_{P_0,\mu}(o)<\Psi_{P,\mu}(o).
$$
This is a contradiction. Therefore, $P$ has exactly $N$ facets.
\end{proof}

\begin{lemma}\label{lemqpa6g}
 Suppose $1<p<\infty$. Let $\alpha_1, \ldots, \alpha_N \in(0, \infty)$, and the unit vectors $u_1, \ldots, u_N(N \geq n+1)$ are in general position in dimension $n$. If there exists a polytope $P \in \mathcal{P}\left(u_1, \ldots, u_N\right)$ satisfying $\eta(P)=o$ and $\tau_{F,p}(P)=1$ such that
$$
\Psi_{\mu,P}(o)=\inf \left\{\max _{\eta \in K} \Phi_K(\eta): K \in \mathcal{P}\left(u_1, \ldots, u_N\right), \tau_{F,p}(K)=1\right\}.
$$
Then there exists a polytope $P_0$ such that
$$
\tau_{F,p}^{\rm{log}}(P_0, \cdot)=\sum_{k=1}^N \alpha_k \delta_{u_k}.
$$
\end{lemma}
\begin{proof}
Since $P \in \mathcal{P}_N\left(u_1, \ldots, u_N\right)$, we choose $|\varepsilon|$ small enough such that
$$
P_\varepsilon=\bigcap_{j=1}^N\left\{x: x \cdot u_j \leq h\left(P, u_j\right)+\varepsilon \delta_j, j=1, \ldots, N\right\},
$$
where $\delta_1, \ldots, \delta_N \in \mathbb{R}$. Let
\begin{align*}
\lambda(\varepsilon)=\tau_{F,p}\left(P_\varepsilon\right)^{-\frac{p-1}{n(p-1)+p}},
\end{align*}
and define $\bar{P}_\varepsilon=\lambda(\varepsilon) P_\varepsilon$. Then $\bar{P}_\varepsilon \in \mathcal{P}_N\left(u_1, \ldots, u_N\right)$, $\tau_{F,p}(\bar{P}_\varepsilon)=1$, and $\bar{P}_\varepsilon \rightarrow P$ with respect to the Hausdorff metric as $\varepsilon \rightarrow 0$.

Let $\eta(\varepsilon):=\eta(\bar{P}_\varepsilon)$,  and set 
$$\Psi(\varepsilon)=\max _{\eta \in \operatorname{Int}(\bar{P}_\varepsilon)} \Psi_{\mu,\bar{P}_\varepsilon}(\eta),$$
 from (\ref{eq:zgx1}) and (\ref{Eq:Fdhs}), it follows that
\begin{align*}
\Psi(\varepsilon) & =\max _{\eta \in \operatorname{Int}(\bar{P}_\varepsilon)} \sum_{k=1}^N \alpha_k \log \left(h\left(\bar{P}_\varepsilon, u_k\right)-\eta \cdot u_k\right) \\
& =\sum_{k=1}^N \alpha_k \log \left(h\left(\bar{P}_\varepsilon,u_k\right)-\eta(\bar{P}_\varepsilon) \cdot u_k\right)\\
& =\sum_{k=1}^N \alpha_k \log \left(\lambda(\varepsilon)h\left(P_\varepsilon,u_k\right)-\eta(\varepsilon) \cdot u_k\right).
\end{align*}
Since the maximum is attained in an interior point $\eta(\varepsilon)$, for each $i=1, \ldots, n$, where $u_{i, k}$ is the $i$-th element of the vector $u_k$, we obtain
\begin{align}\label{Eq:2yd}
\frac{\partial \Psi(\varepsilon)}{\partial \eta(\varepsilon)}=\sum_{k=1}^N \frac{\alpha_k u_{k, i}}{\lambda (t) h\left(P_\varepsilon, u_k\right)-\eta(\varepsilon) \cdot u_k}=0.
\end{align}
 In particular, when $\varepsilon=0$, we have, $\lambda (0)=1, \eta(0)=o$. Thus, by (\ref{Eq:2yd}), it can be deduced that
\begin{align}\label{Eq:auki}
\sum_{k=1}^N \frac{\alpha_k u_{k, i}}{h\left(P, u_k\right)}=0,
\end{align}
for each $i=1, \ldots, n$.

Next, it shows that $\left.\eta^{\prime}(\varepsilon)\right|_{t=0}$ exists.
Let
$$
\Phi_i(\varepsilon, \eta)=\sum_{k=1}^N \frac{\alpha_k u_{k, i}}{\lambda (\varepsilon) h\left(P_\varepsilon, u_k\right)-\left(\eta_1 u_{k, 1}+\ldots+\eta_n u_{k, n}\right)},
$$
for each $i=1, \ldots, n$.

Then we have
$$
\left.\frac{\partial \Phi_i}{\partial \eta_j}\right|_{(0, \ldots, 0)}=\sum_{k=1}^N \frac{\alpha_k u_{k, i} u_{k, j}}{h\left(P, u_k\right)^2},
$$
for each $j=1, \ldots, n$. It follows that
$$
\left(\left.\frac{\partial \Phi}{\partial \eta}\right|_{(0, \ldots, 0)}\right)_{n \times n}=\sum_{k=1}^N \frac{\alpha_k}{h\left(P, u_k\right)^2} u_k u_k^T,
$$
where $u_k u_k^T$ is an $n \times n$ matrix. 

Considering that $u_1, \ldots, u_N$ are not concentrated on any closed hemisphere, for any $x \in \mathbb{R}^n$, with $x \neq\{0\}$, there exists $u_{i_0} \in\left\{u_1, \ldots, u_N\right\}$ satisfies $u_{i_0} \cdot x \neq 0$, and 
$$
x^T\left(\sum_{k=1}^N \frac{\alpha_k}{h\left(P, u_k\right)^2} u_k u_k^T\right) x=\sum_{k=1}^N \frac{\alpha_k\left(x \cdot u_k\right)^2}{h\left(P, u_k\right)^2} \geq \frac{\alpha_{i_0}\left(x \cdot u_{i_0}\right)^2}{h\left(P, u_{i_0}\right)^2}>0.
$$
This shows that $\left.\frac{\partial \Phi}{\partial \eta}\right|_{(0, \ldots, 0)}$ is positively definite. Based on the inverse function theorem, it can be claimed that $\eta^{\prime}(0) = (\eta_1^{\prime}(0), \ldots, \eta_n^{\prime}(0))$ exists. 

Since $\Psi(0)$ is the minimum value of $\Psi(t)$, the result of (\ref{Eq:auki}) yields
\begin{align*}
0 & =\left.\frac{d \Psi(\varepsilon)}{d \varepsilon}\right|_{\varepsilon=0}=\left.\frac{d}{d\varepsilon}\right|_{\varepsilon=0}\left(\sum_{k=1}^N \alpha_k \log \left(\lambda(\varepsilon)h(P_\varepsilon,u_k)-\eta(\varepsilon) \cdot u_k\right)\right) \\
&=\sum_{k=1}^N \alpha_k \left(\frac{\lambda(0)}{h(P,u_k)}\left.\frac{d}{d\varepsilon}\right|_{\varepsilon=0}h(P_\varepsilon,u_k)-\frac{p-1}{n(p-1)+p}\left.\frac{d}{d\varepsilon}\right|_{\varepsilon=0}\tau_{F,p}\left(P_\varepsilon\right)+\frac{\eta'(0)\cdot u_k}{h(P,u_k)}\right)\\
&=\sum_{k=1}^N \alpha_k\left(\frac{\delta_k}{h(P,u_k)}-\frac{p-1}{n(p-1)+p}\left(\sum_{i=1}^N \delta_iS_{F,p}(P,\{u_i\})\right)\right)\\
&=\sum_{k=1}^N \delta_k\left(\frac{\alpha_k}{h(P,u_k)}-\frac{p-1}{n(p-1)+p}\left(\sum_{i=1}^N \alpha_i\right)S_{F,p}(P,\{u_k\})\right).
\end{align*}
Since the  $\delta_i$ ($i=1, \ldots, N$) are arbitrary, we obatin
$$
\alpha_k\frac{1}{h(P,u_k)}=\frac{p-1}{n(p-1)+p}\left(\sum_{i=1}^N \alpha_i\right)S_{F,p}(P,\{v_k\}),
$$
for all $k=1, \ldots, N$. In view of the fact that $P$ is $n$-dimensional and $o \in \operatorname{Int}(P)$, as a result, $h\left(P, u_k\right)>0$, therefore,
$$
\alpha_k=\left(\sum_{i=1}^N \alpha_i\right)\frac{p-1}{n(p-1)+p}h(P, u_k)S_{F,p}(P,\{v_k\}).
$$
Combining (\ref{Eq:logcd}) with (\ref{Eq:sjj2}) yields
$$\left(\sum_{i=1}^N \alpha_i\right) d \tau_{F,p}^{\rm{log}}(P, \cdot)= \mu.$$
Let $P_0=\left(\sum_{i=1}^N \alpha_i\right)^{\frac{p-1}{n(p-1)+p}} P$, then
$$
\tau_{F,p}^{\rm{log}}(P_0, \cdot)=\mu.
$$
Thus, the conclusion holds.
\end{proof}

\begin{proof}[Proof of Theorem \ref{thmlslamp}]
From Lemma \ref{lemqpa6c} and Lemma \ref{lemqpa6g}, we can deduce there exists a polytope $P$ that is a discrete solution to the log-Minkowski problem of anisotropic $p$-torsional measure.
\end{proof}

\subsection{The general case}

Let $\mu$ be a finite Borel measure (not necessarily discrete) on $\mathbb{S}^{n-1}$ that is not concentrated in any closed hemisphere. We first construct a sequence of discrete measures whose support sets are in general position such that the sequence of discrete measures converges to $\mu$ weakly.

Thus, by (\ref{Eq:Fdhs}), for the general measure, we have
\begin{align}\label{Eq:CGFdhs}
\Psi_{\mu, K}(\eta)=\int_{\mathbb{S}^{n-1}}\log\left(h_K(u)-\eta \cdot u\right)d\mu(u).
\end{align}

For each positive integer $k$, we can divide $\mathbb{S}^{n-1}$ into enough small pieces that the diameter of each small piece is less than $\frac{1}{k}$; That is, there are $N_k>0 $ and a partition of $\mathbb{S}^{n-1}$, with $U_{1, k}, \ldots, U_{N_k, k}$ said, makes $d\left(U_{i, k}\right)<\frac{1}{k}$ and $U_{i, k}$ include not empty internal (relative to the topology of $\mathbb{S}^{n-1}$). We can choose $u_{i, k} \in U_{i, k}$ so that $u_{1, k}, \ldots, u_{N_k, k}$ is in general position. When $k$ is large, it is obvious that the vector $u_{1, k}, \ldots, u_{N_k, k}$ cannot be contained in any closed hemisphere.

We define the discrete measure $\mu_k$ on $\mathbb{S}^{n-1}$ by
$$
\mu_k=\sum_{i=1}^{N_k}\left(\mu\left(U_{i, k}\right)+\frac{1}{N_k^2}\right) \delta_{u_{i, k}},
$$
and
\begin{align}\label{Eq:mcddw}
\bar{\mu}_k=\frac{|\mu|}{\left|\mu_k\right|} \mu_k.
\end{align}

It is clear that $\bar{\mu}_k$ is a discrete measure on $\mathbb{S}^{n-1}$ satisfying the conditions in Lemma \ref{lemqpa6g} and $\bar{\mu}_k \rightharpoonup \mu$ weakly. Therefore, by Lemma \ref{lemqpa6g}, there exist polytopes $P_k$ containing the origin in their interiors such that
\begin{align}\label{Eq:pcg9}
\tau_{F,p}^{\rm{log}}(\bar{P}_k, \cdot)=\bar{\mu}_k.
\end{align}

A careful examination of the proofs for Lemma \ref{lemqpa6c} and Lemma \ref{lemqpa6g} immediately reveals that $\bar{P}_k$ is a rescaled version of $P_k$ satisfies $P_k \in \mathcal{P}\left(u_1, \ldots, u_N\right)$ solving (\ref{eq:IMEP}) such that $P_k$ has exactly $N$ facets, $\eta_k(P_k)=o$, $\tau_{F,p}(P_k)=1$ and
\begin{align}\label{Eq:8ydcg}
\Psi_{\bar{\mu},P_k}(o)=\inf \left\{\max _{\eta_k\in K_k} \Psi_{\bar{\mu},K_k}(\eta_k):K_k\in \mathcal{P}\left(u_1, \ldots, u_N\right), \tau_{F,p}(K_k)=1\right\}.
\end{align}

In particular,
\begin{align}\label{Eq:pkhepl}
\bar{P}_k =\left(|\bar{\mu}_k|\right)^{\frac{p-1}{n(p-1)+p}}P_k.
\end{align}

We require the following lemma.
\begin{lemma}[\cite{GLJ2024}]\label{lems6cg}
 Let $u_{1, m}, \ldots u_{N_k, k} \in \mathbb{S}^{n-1}$ be as given above. If
$$
Q_k=\bigcap_{i=1}^{N_k}\left\{x \in \mathbb{R}^n: x \cdot u_{i, k} \leq 1\right\}.
$$
Then for sufficiently large $k$, we have
$$
\mathbb{B}^{n} \subset Q_k \subset 2 \mathbb{B}^{n},
$$
where $\mathbb{B}^{n}$ is the centered unit ball.
\end{lemma}

\begin{lemma}\label{lem59}
 Let $\bar{P}_k$ be as given in (\ref{Eq:pcg9}) and $\eta_k(P_k)$ be the minimizer to (\ref{Eq:8ydcg}) with $\eta_k(P_k)=o$. If $|\mu|=1$ (and consequently $\left|\bar{\mu}_k\right|=1$ ), then there exists $c_0>0$ independent of $k$, such that
\begin{align}\label{Eq:PSIsyj}
\Psi_{\bar{\mu},\bar{P}_k}(o)<c_0.
\end{align}
\end{lemma}
\begin{proof}
Base on Lemma \ref{lems6cg}, for sufficiently large $k$, we have $r \mathbb{B}^{n} \subset r Q_k \subset 2 r \mathbb{B}^{n}$. By the homogeneity of $\tau_{F,p}$, there exists $r_0(k)>0$ such that

$$
\tau_{F,p}(r_0(k) Q_k)=1.
$$

Since $r \mathbb{B}^{n} \subset r Q_k$, we have

$$
r_0(k)^\frac{n(p-1)+p}{p-1} \tau_{F,p}(\mathbb{B}^{n})=\tau_{F,p}(r_0(k) \mathbb{B}^{n}) \leq \tau_{F,p}(r_0(k) Q_k)=1.
$$
Thus, $r_0(k) \leq r_0$ for a constant $r_0$ independent of $k$.
\begin{align*}
& \Psi_{\bar{\mu}_k,P_k}(o)\\
\leq& \int_{\mathbb{S}^{n-1}} \log (h(r_0(k) Q_k,u)-\eta_k(r_0(k) Q_k)\cdot u d \bar{\mu}_k(u) \\
\leq&\int_{\mathbb{S}^{n-1}} \log h(4 r_0 \mathbb{B}^{n},u) d \bar{\mu}_k(u)\\
=&\log\left(4 r_0\right)|\mu|,
\end{align*}
from this and combining (\ref{Eq:8ydcg}) with (\ref{Eq:pkhepl}), it can be obtained that (\ref{Eq:PSIsyj}) is valid.
\end{proof}

\begin{lemma}\label{lem:YXNGD}
 Let $P_m$ be as given in (\ref{Eq:pkhepl}), then there exists $c_1>0$ such that for every $k$, 
$$\tau_{F,p}(\bar{P}_k)>c_1.$$
\end{lemma}
\begin{proof}
From (\ref{Eq:pkhepl}), (\ref{Eq:logcd}), (\ref{Eq:AFNZg}) and the fact $\tau_{F,p}(P_k)=1$, we deduce
$$
\tau_{F,p}(\bar{P}_k)=\left|\tau_{F,p}^{\rm log}(\bar{P}_m, \mathbb{S}^{n-1})\right|=\left|\tau_{F,p}^{\rm log}(\left(|\bar{\mu}_k|\right)^{\frac{p-1}{n(p-1)+p}}P_k, \mathbb{S}^{n-1})\right|=\left|\bar{\mu}_m\right|=|\mu|:=c_1>0.
$$
\end{proof}

A finite Borel measure $\mu$ is said to satisfy the subspace mass inequality if
\begin{align}\label{Eq:ZKJZRL}
\frac{\mu\left(\xi_i \cap \mathbb{S}^{n-1}\right)}{|\mu|}<\frac{i}{n},
\end{align}
for each $i$ dimensional subspace $\xi_i \subset \mathbb{R}^n$ and each $i=1, \ldots, n-1$.

We have an inequality that is weaker than inequality (\ref{Eq:ZKJZRL}), that is assume that the measure $\mu$ satisfies the subspace mass inequality for index $2\leq p<\infty$,
\begin{align}\label{Eq:ZKJZL}
\frac{\mu\left(\xi_i \cap \mathbb{S}^{n-1}\right)}{|\mu|}<1-\frac{(n(p-1)+p)(n-i)}{n(n+2)(p-1)},
\end{align}
for each $i$ dimensional subspace $\xi_i \subset \mathbb{R}^n$, and each $i=1, \ldots, n-1$.

Obviously, if $2\leq p<\infty$, then $\frac{i}{n}\leq1-\frac{(n(p-1)+p)(n-i)}{n(n+2)(p-1)}<1$. If $p=2$, the inequality (\ref{Eq:ZKJZL}) is exactly (\ref{Eq:ZKJZRL}). 

Next, it is to be demonstrated that $P_j$ is uniformly bounded when $\mu$ (not necessarily even) satisfies the subspace mass inequality (\ref{Eq:ZKJZL}). For ease of exposition, we write
$$
\Theta_i=1-\frac{(n(p-1)+p)(n-i)}{n(n+2)(p-1)}=\frac{n(p-2)+i(n(p-1)+p)}{n(n+2)(p-1)}.
$$

For each $\omega \subset \mathbb{S}^{n-1}$ and $\eta>0$, we define
$$
\mathfrak{M}_\zeta(\omega)=\left\{u \in \mathbb{S}^{n-1}:|u-v|<\zeta \text {, for some } v \in \omega\right\}.
$$

The next lemma shows that when $\mu$ satisfies the subspace mass inequality, then the sequence of approximating discrete measures $\bar{\mu}_k$ satisfies a slightly stronger subspace mass inequality for sufficiently large $k$.

\begin{lemma}
 Let $\mu$ be a non-zero finite Borel measure on $\mathbb{S}^{n-1}$ and $\bar{\mu}_k$ be as constructed in (\ref{Eq:mcddw}). If $\mu$ satisfies the subspace mass inequality (\ref{Eq:ZKJZL}), then there exist $\bar{\Theta}_i \in(0, \Theta_i)$, $\zeta_0\in(0,1)$ and $N_0>0$ such that for all $k>N_0$,
\begin{align}\label{Eq:yhbds}
\frac{\bar{\mu}_k\left(\mathfrak{M}_{\zeta_0}\left(\xi_i \cap \mathbb{S}^{n-1}\right)\right)}{|\mu|}<\bar{\Theta}_i,
\end{align}
where $i=1, \ldots, n-1$.
\end{lemma}
\begin{proof}
 Fix $i \in\{1, \ldots, n-1\}$. We first prove that there exist $N_i \in \mathbb{N}$ and $\zeta_i \in(0,1)$ such that for all $k>N_i$,
\begin{align}\label{Eq:Chgco}
\frac{\bar{\mu}_k\left(\mathfrak{M}_{\zeta_i}\left(\xi \cap \mathbb{S}^{n-1}\right)\right)}{|\mu|}<\zeta_i \Theta_i,
\end{align}
for all subspace $\xi$ of $\mathbb{R}^n$ with $\operatorname{dim} \xi=k$.
To prove (\ref{Eq:Chgco}), we argue by contradiction. Then there exists a sequence of $\xi_{i_j}$ with $\operatorname{dim} \xi_{i_j}=k$, a sequence of $\zeta_{i_j} \in(0,1)$ with $\lim _{j \rightarrow+\infty} \zeta_{i_j}=1$, and $k_{i_j}$ such that
$$
\frac{\bar{\mu}_{k_{i_j}}\left(\mathfrak{M}_{\zeta_{i_j}}\left(\xi_{i_j} \cap \mathbb{S}^{n-1}\right)\right)}{|\mu|}>\zeta_{i_j} \Theta_i.
$$
Let $e_{1, i_j}, \ldots, e_{i, i_j}$ be the orthonormal basis of subspace $\xi_{i_j}$, there exists a sequence of $e_1, \ldots, e_k \in \mathbb{S}^{n-1}$ and a subsequence $e_{1, i_j}, \ldots, e_{k, i_j}$ (which is also denoted by $e_{1, i_j}, \ldots, e_{k, i_j}$) such that $e_{1, i_j} \rightarrow e_1, \ldots, e_{k, i_j} \rightarrow e_k$ as $j \rightarrow+\infty$. Let $\xi_0=\operatorname{span}\left\{e_1, \ldots, e_k\right\}$. Thus, for any given small positive constant $\zeta$, we get
$$
\begin{aligned}
\mu_{k_j}\left(\mathfrak{M}_\zeta\left(\xi_0 \cap \mathbb{S}^{n-1}\right)\right) & \geq \mu_{k_j}\left(\mathfrak{M}_{\zeta_{i_j}}\left(\xi_{i_j} \cap \mathbb{S}^{n-1}\right)\right) \\
& >\zeta_{i_j} \Theta_i \mu(\mathbb{S}^{n-1}),
\end{aligned}
$$
for large enough $j$. By letting $j \rightarrow+\infty$, we deduce
$$
\mu\left(\mathfrak{M}_\zeta\left(\xi_0 \cap \mathbb{S}^{n-1}\right)\right) \geq \Theta_i \mu(\mathbb{S}^{n-1}).
$$
Letting $\zeta \mapsto 0$, we obtain
$$
\mu\left(\xi_0 \cap \mathbb{S}^{n-1}\right) \geq \Theta_i \mu(\mathbb{S}^{n-1}),
$$
which is a contradiction. Hence, (\ref{Eq:yhbds}) holds. Set $N_0=\max \left\{N_1, \ldots, N_{n-1}\right\}$ and $\zeta_0=\min \left\{\zeta_1, \ldots, \zeta_{n-1}\right\}$, we deduce
$$
\frac{\bar{\mu}_k\left(\mathfrak{M}_{\zeta_0}\left(\xi \cap \mathbb{S}^{n-1}\right)\right)}{|\mu|}<\zeta_0 \Theta_i=: \bar{\Theta}_i,
$$
whenever $k>N_0$ and $k=1, \ldots, n-1$.
\end{proof}

\begin{lemma}[Lemma 4.1 of \cite{BLYZZ2013}]\label{lem59cg}
Suppose $\alpha_1, \ldots, \alpha_k \in[0,1]$ are such that
$$
\alpha_1+\cdots+\alpha_k=1.
$$
Suppose further that $\beta_1 \leq \cdots \leq \beta_k$ are real numbers. Assume there exist $\gamma_0, \ldots, \gamma_k \in[0, \infty)$, with $\gamma_0=0$, and $\gamma_k=1$, such that
$$
\alpha_1+\cdots+\alpha_l \leq \gamma_l, \quad \text { for } l=1, \ldots, k,
$$
then
$$
\sum_{l=1}^k \alpha_l \beta_l \geq \sum_{k=1}^k\left(\gamma_l-\gamma_{l-1}\right) \beta_l.
$$
\end{lemma}

\begin{lemma}\label{lem:tqgj} 
Let $2\leq p<\infty$, $\mu$ is a finite Borel measure on $\mathbb{S}^{n-1}$ satisfies the subspace mass inequality (\ref{Eq:ZKJZL}). Suppose further that a sequence $\left(e_{1, k}, \ldots, e_{n, k}\right)$, for $k=1,2, \ldots$, of ordered orthonormal bases of $\mathbb{R}^n$, converges to the ordered orthonormal basis $\left(e_1, \ldots, e_n\right)$. If $E_k$ is the sequence of ellipsoids
$$
E_k=\left\{x \in \mathbb{R}^n: \frac{\left(x \cdot e_{1, k}\right)^2}{a_{1, k}^2}+\cdots+\frac{\left(x \cdot e_{n, k}\right)^2}{a_{n, k}^2} \leq 1\right\},
$$
with $0<a_{1, k} \leq a_{2, k} \leq \cdots \leq a_{n, k}$, for all $k$, then there exist $k_0$, and $\delta_0, t_0>0$ such that for each $k>k_0$,
\begin{align*}
\frac{1}{\left|\bar{\mu}_k\right|} \Psi_{\bar{\mu}_k,E_k}\left(o\right) \geq&\log \frac{\delta_0}{2}+t_0\log a_{n, k}+\frac{(1-t_0)n(p-2)}{n(n+2)(p-1)}\log a_{1,k}\\
&+ \frac{(1-t_0)[n(p-1)+p]}{n(n+2)(p-1)}\log |E_k|+\frac{(1-t_0)[n(p-1)+p]}{n(n+2)(p-1)}\log \omega_n.
\end{align*}
\end{lemma}
\begin{proof}
Let $e_1, \ldots, e_n$ be an orthonormal basis in $\mathbb{R}^n$. We define the following partition of the unit sphere. For each $\delta \in\left(0, \frac{1}{\sqrt{n}}\right)$, define
\begin{align}\label{Eq:AD2EC}
A_{l, \delta}=\left\{u \in \mathbb{S}^{n-1}:\left|u \cdot e_l\right| \geq \delta,\left|u \cdot e_j\right|<\delta, \text { for } j>l\right\},
\end{align}
for each $l=1, \ldots, n$. These sets are non-empty since $e_l \in A_{l, \delta}$. They are obviously disjoint. Furthermore, it can be seen that the union of $A_{l, \delta}$ covers $\mathbb{S}^{n-1}$. Indeed, for any unit vector $u \in \mathbb{S}^{n-1}$, by the choice of $\delta$, there has to be at least one $i$ such that $\left|u \cdot e_l\right| \geq \delta$. Let $i_0$ be the largest $l$ that makes $\left|u \cdot e_l\right| \geq \delta$, then $u\in A_{i_0, \delta}$.

Since $\mu$ satisfies the subspace mass inequality (\ref{Eq:ZKJZL}), by (\ref{Eq:yhbds}), there exist $N_0>0$, $ \zeta_0 \in(0,1)$, and $\bar{\Theta}_i \in\left(0, \Theta_i\right)$ such that for all $k>N_0$, (\ref{Eq:yhbds}) holds for each $i$-dimensional proper subspace $\xi_n \subset \mathbb{R}^n$. Let $t_0>0$ be sufficiently small so that
$$
\left(1-t_0\right) \Theta_i>\bar{\Theta}_i.
$$
Therefore, for all $k>N_0$, we obtain
\begin{align}\label{Eq:Thetas}
\frac{\bar{\mu}_k(\mathfrak{M}_{\zeta_0}(\xi_i \cap \mathbb{S}^{n-1}))}{|\mu|}<\left(1-t_0\right) \Theta_i,
\end{align}
for each $i$-dimensional subspace $\xi_i \subset \mathbb{R}^n$ and $i=1, \ldots, n-1$. Specifically, we set $\xi_i=\operatorname{span}\left\{e_1, \ldots, e_i\right\}$.

Notice that for a sufficiently small $\delta_0 \in(0,1)$, we obtain
$$
\bigcup_{l=1}^i A_{l, \delta_0} \subset \mathfrak{M}_{\zeta_0}\left(\xi_i \cap \mathbb{S}^{n-1}\right),
$$
and as a result of (\ref{Eq:Thetas}) and the fact that $A_{l, \delta_0}$ constitutes a partition of $\mathbb{S}^{n-1}$. 

Obviously, if 
$$\alpha_{l, \delta_0}=\frac{\bar{\mu}_k\left(A_{l, \delta_0}\right)}{\left|\bar{\mu}_k\right|},$$
 for each $l=1, \ldots, n$, Then
$$\alpha_{1, \delta_0}+\alpha_{2, \delta_0}+\cdots+\alpha_{n, \delta_0}=1,$$
thus, by the fact that $\left|\bar{\mu}_k\right|=|\mu|$, and $i=1, \ldots, n-1$, we obtain
\begin{align}\label{Eq:C6Gh}
\sum_{l=1}^i\alpha_{l, \delta_0} =\frac{\sum_{l=1}^i \bar{\mu}_k\left(A_{l, \delta_0}\right)}{|\mu|}<\left(1-t_0\right) \Theta_i,
\end{align}

Since $e_{1, k}, \ldots, e_{n, k}$ converges to $e_1, \ldots, e_n$, there exists $N_1>N_0$ such that for each $k>N_1$,
$$
\left|e_{i, k}-e_i\right|<\frac{\delta_0}{2},
$$
 for  $k=1, \ldots, n$. Note that since $\pm a_{i, k} e_{i, k} \in E_k$, we have for each $u \in A_{i, \delta_0}$,
\begin{align}\label{Eq:HEKU}
h(E_k,u) \geq\left|u \cdot e_{i, k}\right| a_{i, k} \geq\left(\left|u \cdot e_i\right|-\left|u \cdot\left(e_{i, k}-e_i\right)\right|\right) a_{i, k} \geq \frac{\delta_0}{2} a_{i, k}.
\end{align}
Therefore, by the fact that $A_{i, \delta_0}$ forms a partition of $\mathbb{S}^{n-1}$ and (\ref{Eq:HEKU}), we obtain
\begin{align*}
\frac{1}{\left|\bar{\mu}_k\right|} \Psi_{\bar{\mu}_k,E_k}\left( o\right) & =\frac{1}{\left|\bar{\mu}_k\right|} \int_{\mathbb{S}^{n-1}} \log h(E_k,u) d \bar{\mu}_k(u) \\
&=\frac{1}{\left|\bar{\mu}_k\right|} \sum_{i=1}^n \int_{A_{i, \delta_0}} \log h(E_k,u) d \bar{\mu}_k(u) \\
& \geq \sum_{i=1}^n \frac{\bar{\mu}_k\left(A_{i, \delta_0}\right)}{\left|\bar{\mu}_k\right|} \log \left(\frac{\delta_0}{2}a_{i, k} \right)\\
& =\log \frac{\delta_0}{2}+\sum_{i=1}^n \alpha_{i, \delta_0} \log a_{i, k}.
\end{align*}
By Lemma \ref{lem59cg}, we further set $\beta_i=\alpha_{1, \delta_0}+\cdots+\alpha_{i, \delta_0}$ for $i=1, \ldots, n$ and $\beta_0=0$. Note that $\beta_n=1$. We have $\alpha_{i, \delta_0}=\beta_i-\beta_{i-1}$ for $i=1, \ldots, n$. Thus,
$$
\begin{aligned}
\sum_{i=1}^n \alpha_{i, \delta_0} \log a_{i, k} & =\sum_{i=1}^n (\beta_i-\beta_{i-1}) \log a_{i, k} \\
& =\log a_{n, k}+\sum_{i=1}^{n-1} \beta_i\left(\log a_{i, k}-\log a_{i+1, k}\right),
\end{aligned}
$$
where in the last equality, we performed summation by parts. Note that by definition of $\beta_i$, equation (\ref{Eq:C6Gh}) simply states $\beta_i<(1-t_0)\Theta_i$, together with the fact that $a_{i, k} \leq a_{i+1, k}$, implies
\begin{align*}
&\nonumber\sum_{i=1}^n \alpha_{i, \delta_0} \log a_{i, k}\\
\geq & \log a_{n, k}+\sum_{i=1}^{n-1} (1-t_0)\Theta_i\left(\log a_{i, k}-\log a_{i+1, k}\right)\\
=& t_0\log a_{n, k}+(1-t_0)\left[\sum_{i=1}^{n-1}\Theta_i\left(\log a_{i, k}-\log a_{i+1, k}\right)+\log a_{n, k}\right]\\
=& t_0\log a_{n, k}+(1-t_0)\left[\Theta_1\log a_{1,k}+\sum_{i=2}^{n-1}(\Theta_i-\Theta_{i-1})\log a_{i, k}+(1-\Theta_{n-1})\log a_{n, k}\right]\\
=& t_0\log a_{n, k}+(1-t_0)\left[\left(\frac{n(p-2)}{n(n+2)(p-1)}+\frac{n(p-1)+p}{n(n+2)(p-1)}\right)\log a_{1,k}\right.\\
&\left.+\sum_{i=2}^{n-1}\frac{n(p-1)+p}{n(n+2)(p-1)}\log a_{i, k}+\frac{n(p-1)+p}{n(n+2)(p-1)}\log a_{n, k}\right]\\
=& t_0\log a_{n, k}+(1-t_0)\left[\frac{n(p-2)}{n(n+2)(p-1)}\log a_{1,k}+ \frac{n(p-1)+p}{n(n+2)(p-1)}\sum_{i=1}^{n}\log a_{i, k}\right]\\
=&t_0\log a_{n, k}+\frac{(1-t_0)n(p-2)}{n(n+2)(p-1)}\log a_{1,k}+ \frac{(1-t_0)[n(p-1)+p]}{n(n+2)(p-1)}\sum_{i=1}^{n}\log a_{i, k}\\
=&t_0\log a_{n, k}+\frac{(1-t_0)n(p-2)}{n(n+2)(p-1)}\log a_{1,k}+ \frac{(1-t_0)[n(p-1)+p]}{n(n+2)(p-1)}\log \left(a_{1, k}a_{1, k}\cdots a_{n-1, k}a_{n, k}\right).
\end{align*}
Since $\left|E_k\right|=\omega_n a_{1, k} a_{2, k} \ldots a_{n, k}$, Thus
\begin{align*}
\frac{1}{\left|\bar{\mu}_k\right|} \Psi_{\bar{\mu}_k,E_k}\left( o\right) \geq&\log \frac{\delta_0}{2}+t_0\log a_{n, k}+\frac{(1-t_0)n(p-2)}{n(n+2)(p-1)}\log a_{1,k}\\
&+ \frac{(1-t_0)[n(p-1)+p]}{n(n+2)(p-1)}\log |E_k|+\frac{(1-t_0)[n(p-1)+p]}{n(n+2)(p-1)}\log \omega_n.
\end{align*}
This concludes its proof.
\end{proof}

Now, we will prove that $\bar{P}_k$ is uniformly bounded.
\begin{lemma}\label{Lem:YzyjgJ}
 Let $2\leq p<\infty$, Suppose $\mu$ is a non-zero finite Borel measure on $\mathbb{S}^{n-1}$ and $\bar{\mu}_k$ is defined as in (\ref{Eq:mcddw}). Let $\bar{P}_k$ be constructed as in (\ref{Eq:pkhepl}). If $\mu$ satisfies the subspace mass inequality (\ref{Eq:ZKJZL}), then $\bar{P}_k$ is uniformly bounded.
\end{lemma}
\begin{proof}
We use the method of contradiction. Assume that $\bar{P}_k$ is not uniformly bounded. Let $E_k$ denote the John ellipsoid of $\bar{P}_k$, and then
\begin{align}\label{Eq:YHTQ}
E_k \subset \bar{P}_k \subset n\left(E_k-o_k\right)+o_k,
\end{align}
where the ellipsoid $E_k$ centered at $o_k \in \operatorname{Int}\left(\bar{P}_k\right)$ is characterized by
$$
E_k=\left\{x \in \mathbb{R}^n: \frac{\left|\left(x-o_k\right) \cdot e_{1,k}\right|^2}{a_{1,k}^2}+\cdots+\frac{\left|\left(x-o_k\right) \cdot e_{n,k}\right|^2}{a_{n,k}^2} \leq 1\right\},
$$
for a sequence of ordered orthonormal bases $\left(e_{1,k}, \ldots, e_{n,k}\right)$ of $\mathbb{R}^n$, with $0<a_{1, l} \leq \cdots \leq a_{n, l}$. Since $\bar{P}_k$ is not uniformly bounded, by taking a subsequence, we can assume that $a_{n,k} \rightarrow \infty$ and $a_{n,k} \geq 1$. Due to the compactness of $\mathbb{S}^{n-1}$, we can select a subsequence and assume that $\left\{e_{1,k}, \ldots, e_{n,k}\right\}$ converges to $\left\{e_1, \ldots, e_n\right\}$ which forms an orthonormal basis in $\mathbb{R}^n$. From (\ref{Eq:sjgj}), Lemma \ref{lem:tqgj}, (\ref{Eq:YHTQ}), there exists $\delta_0, t_0, c_n$ and $N_0>0$ such that
\begin{align}\label{Eq:tygjshE}
\nonumber&\frac{1}{\left|\bar{\mu}_k\right|} \Psi_{ \bar{\mu}_k,\bar{P}_k}\left(o_k\right)\\
\nonumber\geq&\frac{1}{\left|\bar{\mu}_k\right|} \Psi_{\bar{\mu}_k,E_k}\left(o_k\right)=\frac{1}{\left|\bar{\mu}_k\right|} \Psi_{\bar{\mu}_k,E_k-o_k}(o) \\
\nonumber\geq& \log \frac{\delta_0}{2}+t_0\log a_{n, k}+\frac{(1-t_0)n(p-2)}{n(n+2)(p-1)}\log a_{1,k}\\
\nonumber &+ \frac{(1-t_0)[n(p-1)+p]}{n(n+2)(p-1)}\log |E_k|+\frac{(1-t_0)[n(p-1)+p]}{n(n+2)(p-1)}\log \omega_n\\
\nonumber\geq& \log \frac{\delta_0}{2}+t_0\log a_{n, k}+\frac{(1-t_0)n(p-2)}{n(n+2)(p-1)}\log a_{1,k}+ \frac{1-t_0}{n+2}\log \tau_{F,p}(E_k)\\
  &+\frac{(1-t_0)[n(p-1)+p]}{n(n+2)(p-1)}\log \omega_n-\frac{1-t_0}{n+2}\log\left(\frac{p-1}{n(p-1)+p}n^{-\frac{1}{p-1}}\kappa_n^{-\frac{p}{n(p-1)}}\right).
\end{align}
By using the homogeneity, monotonicity, and translation invariance of $\tau_{F,p}$ and (\ref{Eq:YHTQ}), we get
\begin{align}\label{Eq:po099}
\nonumber \tau_{F,p}(E_k) & =\tau_{F,p}(E_k-o_k)=n^{-\frac{n(p-1)+p}{(p-1)}} \tau_{F,p}(n(E_k-o_k )+o_k ) \\
& \geq n^{-\frac{n(p-1)+p}{(p-1)}}\tau_{F,p}(\bar{P}_k )=n^{-\frac{n(p-1)+p}{(p-1)}}|\bar{\mu}_k|.
\end{align}

Recall that $\eta_k(P_k)$ is the minimizer to (\ref{Eq:8ydcg}) with $\eta_k(P_k)=o$, by applying (\ref{Eq:pkhepl}), we know that $\eta_k(\bar{P}_k)=o$, this combines with $\left|\bar{\mu}_k\right|=|\mu|$, (\ref{Eq:po099}) and that $a_{n,k} \rightarrow \infty$, from (\ref{Eq:tygjshE}), we have
$$
\Psi_{\bar{\mu}_k,\bar{P}_k}(o) \geq \Psi_{\bar{\mu}_k,\bar{P}_k}\left(o_k\right) \rightarrow \infty, \text { as } j \rightarrow \infty,
$$
which contradicts to Lemma \ref{lem59}. Thus, the proof is completed.
\end{proof}

\begin{lemma}\label{Lem:cgteq}
If $\bar{P}_k$ in (\ref{Eq:pcg9}) are uniformly bounded and $\tau_{F,p}(\bar{P}_k)>c_0$ for some constant $c_0>0$, then there exists a convex body $K \in \mathscr{K}^n$ with $o \in K$ such that
\begin{align}\label{Eq:pcgg}
\tau_{F,p}^{\rm{log}}(K, \cdot)=\mu.
\end{align}
\end{lemma}

\begin{proof}
 By the Blaschke selection theorem \cite[Theorem 1.8.7]{SRA2014}, there exists a subsequence $\bar{P}_{k_j}$ of $\bar{P}_k$ such that $\bar{P}_{k_j} \rightarrow K$ as $j \to \infty$, where $K$ is a compact convex set containing the origin. Due to the continuity of $\tau_{F,p}$ and the fact that $\tau_{F,p}(\bar{P}_k)>c_1$, we obtain $\tau_{F,p}(K)>0$. From (\ref{Eq:sjgj}), we have 
\begin{align*}
\tau_{F,p}(K) \leq\frac{p-1}{n(p-1)+p}n^{-\frac{1}{p-1}}\kappa_n^{-\frac{p}{n(p-1)}}|K|^{\frac{n(p-1)+p}{n(p-1)}},
\end{align*}
Consequently, we can derive
\begin{align*} 
|K|\geq\left(\frac{n(p-1)+p}{p-1}n^{\frac{1}{p-1}}\kappa_n^{\frac{p}{n(p-1)}}\right)^{\frac{n(p-1)}{n(p-1)+p}}\tau_{F,p}(K)>0.
\end{align*}
Thus, $K$ is non-degenerate, which in turn implies $K$ has nonempty interior. Equation (\ref{Eq:pcgg}) now easily follows by taking the limit of both sides of (\ref{Eq:pcg9}) and applying Lemma \ref{lemapytrs}.
\end{proof}

\begin{proof}[Proof of Theorem \ref{thmyblmp}] Based on Lemma \ref{lem:YXNGD}, Lemma \ref{Lem:YzyjgJ}, and Lemma \ref{Lem:cgteq}, we can conclude that Theorem \ref{thmyblmp} holds.
\end{proof}

\end{sloppypar}


\begin{thebibliography}{10}

\bibitem{AGHLV2022} M. Akman, J. Gong, J. Hineman, J. Lewis and A. Vogel, \emph{The Brunn-Minkowski inequality and a Minkowski problem for nonlinear capacity}, Mem. Amer. Math. Soc., {\bf 275} (2022), 1348.

\bibitem{AKM2023} M. Akman and S. Mukherjee, \emph{On the Minkowski problem for $p$-harmonic measures},  Calc. Var. Partial Differential Equations, {\bf 64} (2025), 1: 36.

\bibitem{AAD1938} A. D. Alexandrov, \emph{On the theory of mixed volumes of convex bodies, III: Extension of two theorems of Minkowski on convex polyhedra to arbitrary convex bodies}, Mat. Sb. (N.S.), {\bf 3}
(1938), 27-46, Selected works I, 99-117.

\bibitem{ABEND2001} B. Andrews, \emph{Volume-preserving anisotropic mean curvature flow}, Indiana Univ. Math. J., {\bf 50} (2001), 2: 783-828.

\bibitem{BMT2024} V. Balestro, H. Martini and R. Teixeira, \emph{Convexity from the Geometric Point of View}, Springer Nature, 2024.

\bibitem{BEP1996} G. Bellettini and M. Paolini, \emph{Anisotropic motion by mean curvature in the context of Finsler geometry}, Hokkaido Math. J., {\bf 25} (1996), 3: 537-566.

\bibitem{BCG2018} C. Bianchin and G. Ciraolo, \emph{Wulff shape characterizations in overdetermined anisotropic elliptic problems}, Comm. Partial Differential Equations, {\bf 43} (2018), 5: 790-820.

\bibitem{BCS2016} C. Bianchini, G. Ciraolo, and P. Salani, \emph{An overdetermined problem for the anisotropic capacity}, Calc. Var. Partial Differential Equations, {\bf 55} (2016), 84.

\bibitem{BLYZZ2013} K. J. B\"{o}r\"{o}czky, E. Lutwak, D. Yang, G. Zhang and Y. Zhao, \emph{The dual Minkowski problem for symmetric convex bodies}, Adv. Math., {\bf 365} (2019), 106805.

\bibitem{BRF2024} K. J. B\"{o}r\"{o}czky, J. P. G. Ramos and A. Figalli, The Isoperimetric inequality, the Brunn-Minkowski theory and Minkowski type Monge-Amp\`ere equations on the sphere, 2024, https://users.renyi.hu/~carlos/Brunn-Minkowski-Book-2024-02-05.pdf.

\bibitem{CHD2020} Z. Chen and Q. Dai, \emph{The $L_p$ Minkowski problem for torsion}, J. Math. Anal. Appl., {\bf 488}, (2020), 124060.

\bibitem{CCPS2009} A. Cianchi and P. Salani, \emph{Overdetermined anisotropic elliptic problems}, Math. Ann., {\bf 345} (2009),859-881.

\bibitem{CSP2023} A. Cianchi and P. Salani, \emph{Wulff shape symmetry of solutions to overdetermined problems for Finsler Monge-Amp\`{e}re equations}, J. Funct. Anal., {\bf 285} (2023), 9: 110091.

\bibitem{FNS2020} G. Ciraolo, A. Figalli and A. Roncoroni, \emph{Symmetry results for critical anisotropic $p$-Laplacian equations in convex cones}, Geom. Funct. Anal., {\bf 30} (2020), 770-803.

\bibitem{CGL2020} G. Ciraolo and X. Li, \emph{An exterior overdetermined problem for Finsler $N$-Laplacian in convex cones}, Calc. Var. Partial Differential Equations, {\bf 61} (2022), 4: 121.


\bibitem{COA2005} A. Colesanti, \emph{Brunn-Minkowski inequalities for variational functionals and related problems}, Adv. Math., {\bf 194} (2005), 105-140.

\bibitem{CAM2010} A. Colesanti and M.Fimiani, \emph{The Minkowski problem for torsional rigidity}, Indiana Univ. Math. J., {\bf 59} (2010), 3: 1013-1039.

\bibitem{DQY2025} Q. Dai and X. Yi, \emph{Minkowski problems arise from sub-linear elliptic equations}, J. Differential Equations, {\bf 415} (2025), 764-790.

\bibitem{FJ1938} W. Fenchel and B. Jessen, \emph{Mengenfunktionen und konvexe K\"{o}rper}, Danske Vid. Selsk. Mat.-Fys. Medd. {\bf 16} (1938), 1-31.

\bibitem{FNS2018} D. Francesco, G. Nunzia and G. Serena, \emph{On functionals involving the torsional rigidity related to some classes of nonlinear operators}, J. Differential Equations, {\bf 265} (2018), 12: 6424-6442.

\bibitem{GRAM024} A.Greco and B. Mebrate, \emph{An overdetermined problem related to the Finsler $p$-Laplacian}, Mathematika,{\bf 70} (2024), 4: e12267.
 
\bibitem{GLJ2024} L. Guo, D. Xi and Y. Zhao, \emph{The $L_p$ chord Minkowski problem in a critical interval}, Math. Ann., {\bf 389} (2024), 3: 3123-3162.

\bibitem{HLYZ2010} C. Haberl, E. Lutwak, D. Yang and G. Zhang, \emph{The even Orlicz Minkowski problem}, Adv. Math., {\bf 224} (2010), 2485-2510.	

\bibitem{HJR2024} J. Hu, \emph{A Gauss curvature flow approach to the torsional Minkowski problem}, J. Differential Equations, {\bf 385} (2024), 254-279.

\bibitem{HJR20241} J. Hu, \emph{The torsion log-Minkowski Problem}, J. Geom. Anal., {\bf 34} (2024). DOI: 10.1007/s12220-024-01670-1.

\bibitem{HLYZ2016} Y. Huang, E. Lutwak, D. Yang and G. Zhang, \emph{Geometric measures in the dual Brunn-Minkowski theory and their associated Minkowski problems}, Acta Math., {\bf 216} (2016), 325-388.

\bibitem{JAR2014} J. Jaro\v{s}, \emph{Picone's identity for a Finsler $p$-Laplacian and comparison of nonlinear elliptic equations}, Math. Bohem., {\bf 139} (2014), 3: 535-552.

\bibitem{JER1989} D. Jerison, \emph{Harmonic measure in convex domains}, Bull. Amer. Math. Soc. (N.S.), {\bf 21} (1989), 2: 255-260.

\bibitem{JER1991} D. Jerison, \emph{Prescribing harmonic measure on convex domains}, Invent. Math., {\bf 105} (1991), 1: 375-400.

\bibitem{JER1996} D. Jerison, \emph{A Minkowski problem for electrostatic capacity}, Acta Math., {\bf 176} (1996), 1: 1-47.

\bibitem{LEE2021} M. Kim and T. Lee, \emph{The discrete logarithmic Minkowski problem for the electrostatic $\mathfrak{p}$-capacity}, arXiv:2111.07321, 2021.

\bibitem{KLL2023} L. Kryvonos and D. Langharst, \emph{Measure theoretic Minkowski's existence theorem and projection bodies}, Trans. Amer. Math. Soc., {\bf 376} (2023), 8447-8493.

\bibitem{LCZX2024} C. Li and X. Zhao, \emph{Flow by Gauss curvature to the Minkowski problem of $p$-harmonic measure}, arXiv:2404.18757v2, 2024.

\bibitem{LNYZ2020} N. Li and B. Zhu, \emph{The Orlicz-Minkowski problem for torsional rigidity}, J. Differential Equations, {\bf 269} (2020), 10: 8549–8572.
 
\bibitem{LUTWA1993} E. Lutwak, \emph{The Brunn-Minkowski-Firey theory I: Mixed volumes and the Minkowski problem}, J. Differential Geom., {\bf 38} (1993), 1: 131-150.

\bibitem{LOLI1995} E. Lutwak, and V. Oliker, \emph{On the regularity of solutions to a generalization of the Minkowski problem}, J. Differential Geom., {\bf 38} (1995) 1: 227-246.

\bibitem{MED2018} D. Medkov\'{a},  \emph{The Laplace equation. Boundary value problems on bounded and unbounded Lipschitz domains}, Springer, Cham, 2018.


\bibitem{PDG2014} F. D. Pietra and N. Gavitone, \emph{Sharp bounds for the first eigenvalue and the torsional rigidity related to some anisotropic operators}, Math. Nachr., {\bf 287} (2014), 194-209.
 
\bibitem{SRA2014} R. Schneider, \emph{Convex Bodies: The Brunn-Minkowski theory}, 2nd edn, Cambridge Univ. Press, Cambridge, 2014.

\bibitem{SHZ2001} Z. Shen,  \emph{Lectures on Finsler Geometry}, World Scientific Publishing Co., Singapore, 2001.
 
\bibitem{SXZ2021} G. Sun, L. Xu and P. Zhang, \emph{The uniqueness of the $L_p$ Minkowski problem for $q$-torsional rigidity}, Acta Math. Sci., {\bf 41} (2021), 1405-1416.

\bibitem{WGXZ2011} G. Wang and C. Xia \emph{A characterization of the Wulff shape by an overdetermined anisotropic PDE}, Arch. Rational Mech. Anal., {\bf 199} (2011) 99-115.

\bibitem{WGX2012} G. Wang and C. Xia \emph{An optimal anisotropic Poincar\'e inequality for convex domains}, Pacific J. Math., {\bf 258} (2012), 2: 305-326.

\bibitem{XC2013} C. Xia, \emph{On an anisotropic minkowski problem}, Indiana Univ. Math. J., {\bf 62} (2013), 5: 1399-1430.

\bibitem{XCY2022} C. Xia and J. Yin, \emph{Two overdetermined problems for anisotropic $p$-Laplacian}, Math. Eng., {\bf 4} (2022), 2: 1-18.

\bibitem{YH2014} S. T. Yin and Q. He, \emph{The first eigenvalue of Finsler $p$-Laplacian}, Differential Geom. Appl., {\bf 35} (2014), 30-49.

\bibitem{ZGX2014} G. Zhu, \emph{The logarithmic Minkowski problem for polytopes}, Adv. Math., {\bf 262} (2014), 909-931.
 
\bibitem{ZDX2020} D. Zou and G. Xiong, \emph{The $L_p$ Minkowski problem for the electrostatic $\mathfrak{p}$-capacity}, J. Differential Geom., {\bf 116} (2020), 3: 555-596.

\end{thebibliography}
\end{document}